\documentclass[11pt,amsfonts]{article}
\usepackage[utf8]{inputenc}
\usepackage[T1]{fontenc} 
\usepackage{tikz}
\usepackage[english]{babel} 
\usepackage{mathrsfs}
\usepackage{amssymb}
\usepackage[cmex10]{amsmath}
\usepackage{amsthm}
\usepackage{wrapfig}
\usepackage{latexsym}
\usepackage{amsfonts}
\usepackage{color}
\usepackage[font=small, labelfont=bf]{caption}
\usepackage{enumitem}
\usepackage{ctable}
\usepackage{floatrow}
\usepackage[colorlinks=true, citecolor=red, linkcolor=blue]{hyperref}
\usepackage{graphicx}
\usepackage{subfig}
\usepackage{verbatim}
\usepackage{epstopdf}
\usepackage{float}
\usepackage{appendix}

\usepackage[top=2cm , bottom=2cm , left=2.5cm ,right=2.5cm ]{geometry}

\usepackage{dsfont}
\usepackage{amsmath, amssymb}

\usepackage{cases}

\theoremstyle{plain}
\newtheorem{definition}{Definition}
\newtheorem{theorem}{Theorem}

\newtheorem{proposition}{Proposition}

\newtheorem{lemma}{Lemma}
\newtheorem{remark}{Remark}
\definecolor{assumplabelcolor}{RGB}{0,150,0}

\newtheoremstyle{assumpstyle}
  {3pt}{3pt}{\itshape}{}{\bfseries}{.}{0.5em}
  {\thmname{#1}~{\color{assumplabelcolor}\thmnote{#3}}}

\theoremstyle{assumpstyle}
\newtheorem{assumptionmanual}{Assumption}
\newenvironment{assumption}[1]
  {\renewcommand\theassumptionmanual{#1}\assumptionmanual[#1]}
  {\endassumptionmanual}

\newlist{assumpdesc}{description}{1}
\setlist[assumpdesc]{labelwidth=1cm, leftmargin=!, labelindent=0cm}

\newcommand{\assumptionitem}[2]{%
  \item[\textbf{(#1)}]\label{#2}%
}

\numberwithin{equation}{section}
\numberwithin{theorem}{section}
\numberwithin{proposition}{section}
\numberwithin{definition}{section}
\numberwithin{remark}{section}
\numberwithin{corollary}{section}
\numberwithin{lemma}{section}
\numberwithin{remark}{section}

\usepackage{graphicx}

\makeatletter
\newcommand*\bigcdot{\mathpalette\bigcdot@{.5}}
\newcommand*\bigcdot@[2]{\mathbin{\vcenter{\hbox{\scalebox{#2}{$\m@th#1\bullet$}}}}}
\makeatother
\begin{document}
	\sloppy
	
	\begin{center}
		{\Large \textbf{Mean Field Games with Reflected Dynamics}} \\[0pt]
		~\\[0pt]  Ayoub Laayoun, Imane Jarni and Badr Missaoui 
        \renewcommand{\thefootnote}{}

        \footnotetext{Moroccan Center for Game Theory, Mohammed VI Polytechnic University, Rabat, Morocco.}
        \footnotetext{E-mails: \text{ayoub.laayoun@um6p.ma, imane.jarni@um6p.ma, badr.missaoui@um6p.ma}}
        	\renewcommand{\thefootnote}{\arabic{footnote}}

	\end{center}
	
	\renewcommand{\thefootnote}{\arabic{footnote}}
	
\begin{abstract}
This paper establishes the existence of an equilibrium for a class of mean field games governed by reflected stochastic differential equations. The analysis is developed within the framework of relaxed controls and the associated martingale problem, providing a flexible approach that ensures the compactness and continuity properties required for the existence proof. Under a suitable convexity assumption, we recover an equilibrium for the original (non-relaxed) mean field game. Furthermore, by imposing a uniform ellipticity condition, we establish the existence of a Markovian equilibrium.
\end{abstract}

\section{Introduction }
The study of Mean Field Games (MFGs), introduced in \cite{malhame,lasry3, lasry2, LasryP}, offers a powerful framework for analyzing approximate Nash equilibria in large, symmetric population games where players interact only through the empirical distribution of all players' states or strategies. This approach, as described in \cite{Bookdela}, involves approximating the empirical distribution with an externally specified distribution and then examining the optimization problem of a representative player. The next step is to identify a fixed point where the distribution of the representative player’s optimal state or strategy matches the  specified distribution. 

In this paper, we study a class of MFGs with reflection, based on the standard fixed-point formulation.\\
Let $\mu = (\mu_t)_{t \in [0,T]}$ be a deterministic flow of probability measures in $\mathcal{P}(\mathbb{R}_+)$. For a given $\mu$, we consider the stochastic control problem consisting in minimizing the following cost function over all admissible controls $(u_t)_{t\in[0,T]}$.
\begin{equation}\label{control-point} 
\mathbb{E}\Bigg[
\int_0^T f(t,X_t,\mu_t,u_t)\,\mathrm{d}t
+ \int_0^T h(t,X_t,\mu_t)\,\mathrm{d}K_t + g(X_T,\mu_T)
\Bigg],
\end{equation}
where the state process satisfies the reflected dynamics
\begin{equation*}
\begin{cases}
\mathrm{d}X_t = b(t,X_t,\mu_t,u_t)\,\mathrm{d}t
+ \sigma(t,X_t,\mu_t,u_t)\,\mathrm{d}B_t
+ \mathrm{d}K_t, \\[6pt]
X_0 \sim \lambda, \\[4pt]
X_t \ge 0 \quad \text{a.s. for all } t \in [0,T], \\[6pt]
\displaystyle \int_0^T X_t \,\mathrm{d}K_t = 0 
\quad \text{a.s.}
\end{cases}
\end{equation*}
Here $K$ is an adapted, non-decreasing process with $K_0 = 0$, enforcing the state constraint through the Skorokhod condition.\\
A mean field equilibrium is a pair $(\mu,u)$ such that $u$ is optimal for the control problem associated with $\mu$, and the consistency condition
\begin{equation}\label{mean field-point}
\mu_t = Law(X_t), \qquad t \in [0,T],
\end{equation}
is satisfied, where $X$ denotes the corresponding optimal state process.

The existence of mean field game equilibria in the setting without reflection has been established using three main approaches. The first is the analytical approach introduced by Lasry and Lions in \cite{LasryP}, which is based on the study of a coupled system of forward--backward partial differential equations. The backward equation is the Hamilton--Jacobi--Bellman equation associated with the representative agent’s optimization problem, while the forward equation is a Kolmogorov--Fokker--Planck equation describing the evolution of the state distribution arising from the consistency condition.

A second line of work adopts a probabilistic approach based on the Pontryagin maximum principle. In this framework, the optimal control problem is characterized through adjoint processes, and the mean field equilibrium is obtained by solving a system of McKean--Vlasov forward--backward stochastic differential equations. We refer, for instance, to \cite{probaappr1, probaappr2, delanaly} for developments along these lines.

A third approach, proposed by Lacker in \cite{laker}, relies on a weak formulation of stochastic control problems. This method allows for relaxed controls and is formulated in terms of martingale problems, providing a game-theoretic interpretation of MFGs. Using this relaxed solution concept, Lacker proves existence by first establishing the upper hemicontinuity of the representative agent’s best-response correspondence with respect to a given flow of measures $\mu$, via Berge’s maximum theorem. He then applies the Kakutani--Fan--Glicksberg fixed-point theorem to show the existence of a measure-valued process $\mu$ such that the law of the state process, under an optimal response to $\mu$, coincides with $\mu$ itself.

In the reflection setting, Bayraktar et al.\cite{Bayraktar-MFG} developed a mean field game formulation for a large symmetric queueing system operating under heavy-traffic conditions, where numerous strategic servers interact weakly through the empirical distribution of their queue lengths. The authors analyzed this model using the first MFG approach, assuming that the volatility is nondegenerate and independent of both the mean field and the control variables. In the absence of mean field interactions---namely, in the framework of a controlled reflected stochastic differential equation (RSDE)---related problems have been thoroughly examined in several works (see, e.g.,~\cite{Bayraktar-cont,Kushner-Variance, Kushner-Book, laayoun-control, Ref-cont-1}).

Extending the results of~\cite{laker}, we investigate mean field games (MFGs) with reflection as described in~\eqref{control-point} and~\eqref{mean field-point}. Motivated by Lacker’s methodology, our analysis employs the relaxed control framework for stochastic optimal control problems to study the reflected MFG setting. The proposed approach builds upon the seminal contributions of El Karoui et al.~\cite{jeanbl} and the general martingale problem framework introduced by Haussmann and Lepeltier~\cite{Haussmann-lepeltier}, integrating these techniques to address the reflection component within the MFG structure.

The remainder of the paper is organized as follows. 
Section~\ref{Assumptions and main results} recalls the notion of relaxed controls for stochastic control problems governed by systems of reflected stochastic differential equations (SDEs) and introduces the framework of MFGs with reflection. 
This section also presents the standing assumptions on the coefficients and states the main existence results, namely Theorems~\ref{theomfg} and~\ref{theorem:Markovian-new}, together with the proof of the latter. 
Finally, Section~\ref{proof of theorem-mfg} is devoted to the proof of Theorem~\ref{theomfg}.

\section{Assumptions and the main result}\label{Assumptions and main results}
For a measurable space $(\Omega, \mathcal{F})$, let $\mathcal{P}(\Omega)$ denote the set of probability measures on $(\Omega, \mathcal{F})$. When $\Omega$ is a topological space, let $\mathcal{B}(\Omega)$ denote its Borel $\sigma$-algebra, and endow $\mathcal{P}(\Omega)$ with the topology of weak convergence.

For a finite time horizon $T > 0$, we denote by $\mathcal{C}[0,T]$ the space of continuous functions from $[0, T]$ to $\mathbb{R}$, and by $\mathcal{A}[0,T] \subset \mathcal{C}[0,T]$ the subset of non-decreasing functions. We endow $\mathcal{C}[0,T]$ with the supremum norm $\|\cdot\|_T$, defined by
$$
\|x\|_t := \sup_{s \in [0,t]} |x_s|, \quad t \in [0, T], \quad x \in \mathcal{C}[0,T].
$$
For $\mu \in \mathcal{P}(\mathcal{C}[0,T])$, denote by $\mu_t$ the image of $\mu$ under the map $\mathcal{C}[0,T] \ni x \mapsto x_t \in \mathbb{R}$.

For $p \geq 0$ and a separable metric space $(E, d)$, let $\mathcal{P}_p(E)$ denote the set of $\mu \in \mathcal{P}(E)$ satisfying
$$
\int_E d^p(x, x_0) \mu(dx) < \infty,
$$
for some (and thus for any) $x_0 \in E$. For $p \geq 1$ and $\mu, \nu \in \mathcal{P}_p(E)$, let $W_{E,p}$ denote the $p$-Wasserstein distance, given by
\begin{equation}\label{Wassersteindis}
W_{E,p}(\mu, \nu) := \inf_{\gamma \in \mathcal{P}(E \times E)} \left( \int_{E \times E} d^p(x, y) \gamma(dx, dy) \right)^{1/p},
\end{equation}
where $\gamma$ has marginals $\mu$ and $\nu$.

Unless otherwise stated, the space $\mathcal{P}_p(E)$ is always equipped with the metric $W_{E,p}$. If $E$ is complete and separable, so is $(\mathcal{P}_p(E), W_{E,p})$.

Given $p \geq 0$ and $\mu \in \mathcal{P}(\mathbb{R})$, we will frequently use the abbreviation
$$
|\mu|^p := \int_{\mathbb{R}} |x|^p \mu(dx).
$$
Similarly, for $\mu \in \mathcal{P}(\mathcal{C}[0,T])$, we write
$$
\|\mu\|^p_{t} := \int_{\mathcal{C} [0,T]}\|x\|_t^p \mu(dx).
$$
\subsection{The Stochastic control problem}
We consider the following stochastic optimal control problem: 
\begin{equation} \label{cost1}
\inf_{u} \mathbb{E}\left(\int_{0}^{T} f(t,X_{t},u_{t})\mathrm{d}t + \int_{0}^{T} h(t, X_t) \mathrm{d}K_{t} +g(X_T),\right),
\end{equation}
subject to 
\begin{equation} \label{eq01.1}
\begin{cases}
X_t =X_0+\displaystyle\int_0^t b(s, X_s, u_s)\mathrm{d}s + \displaystyle\int_0^t \sigma(s, X_s,u_s) \mathrm{d}B_s +  K_t, & t \in [0, T]; \\
X_0  \sim \lambda;\\
X_t \geq 0 \quad \text{a.s}; \\
\displaystyle\int_0^T X_t \mathrm{d}K_t = 0 \quad \text{a.s.},
\end{cases}
\end{equation}
where $(K_{t})_{t \in [0,T]}$ is non-decreasing process, the real-valued process $(X_{t})_{ t \in [0,T]}$ represents the state process and, and $u$ belongs to the set of controls, which consists of adapted processes taking values in a compact subset $U\subset\mathbb{R}^d$, $d\geq 1$.\\
We will use the compactification method to study the control problem $\eqref{cost1}$, where the state is given by the reflected SDE \eqref{eq01.1}. By introducing a relaxed version of the classical control, we reformulate the problem as a martingale problem on a suitable canonical space.

Now, we introduce the concept of control for the system described by \eqref{eq01.1}.
\begin{definition} \label{def of control}
We call an admissible system control a set $\alpha = (\Omega,\mathcal{F},\mathbb{F},\mathbb{P},B,u,X,K)$ such that 
\begin{description}
\item[1)]  $(\Omega,\mathcal{F},\mathbb{P})$ is a probability space equipped with a filtration $\mathbb{F}=(\mathcal{F}_t)_{t\geq 0}$;
\item[2)] $u$ is a $U$-valued process, $\mathbb{F}$-progressively measurable;
\item[3)] $K$ is $\mathbb{R}_{+}$-valued process, $\mathbb{F}$-adapted; the sample paths of $K$ are in $\mathcal{A}[0,T]$,i.e., for each $\omega \in \Omega$, $K_{.}(\omega) \in \mathcal{A}[0,T]$;
\item[4)] $B$ is a standard Brownian motion on $(\Omega, \mathcal{F}, \mathbb{F}, \mathbb{P})$, and $X$, the state process, is $\mathbb{F}$-adapted with sample paths in $\mathcal{C}[0,T]$;
\item[5)] $\mathbb{P}\circ(X_0)^{-1}= \lambda$;
\item[6)] The processes $X$ and $K$ satisfy the system described by \eqref{eq01.1}.
\end{description}
 The collection of controls is denoted by $\Lambda$.
\end{definition}
 It is well known from the theory of reflected SDEs \cite{andreyg, skorefl} that, under the Lipschitz condition, the set $\Lambda$ is nonempty.\\
    The cost functional corresponding to a control $\alpha \in \Lambda$ is defined as:
\begin{equation}
J(\alpha):= \mathbb{E}^{\mathbb{P}} \left( \int_{0}^{T} f(t,X_{t},u_{t})\mathrm{d}t + \int_{0}^{T} h(t, X_t)\mathrm{d}K_{t} +g(X_T)\right).
\end{equation}
\subsubsection{Relaxed control}
In order to apply the compactification method, we now reformulate the problem. Since the Brownian motion in the definition of controls is not known in advance, we can reformulate the above control problem as an equivalent martingale problem. This simplifies taking limits. In fact, let 
\begin{equation} \label{operator}
   \mathcal{L} \phi(s,x,u):= \dfrac{1}{2}\sigma^{2}(s,x,u) \phi''(x)+ b(s,x,u)\phi'(x).
\end{equation}
Then we can show that $\alpha$ is a control if and only if $\alpha$ satisfies $\textbf{(1)-(3)}$, and $\textbf{(5)}$ in Definition \ref{def of control}, and condition $\textbf{(6)}$ is replaced by
\begin{description}
\item[6$^{\prime})$] The process $X$ is continuous, $\mathbb{F}$-adapted process, such that 
\begin{itemize}
\item[i)] For each $\phi \in \mathcal{C}^{2}_{b}(\mathbb{R})$, $\mathcal{M}^{\phi}$ is $(\mathbb{F},\mathbb{P})$-continuous martingale, where
\begin{equation} \label{equ2.5}
\mathcal{M}^{\phi}_{\cdot} := \phi(X_{\cdot}) - \int_{0}^{\cdot} \mathcal{L} \phi(s,X_{s},u_{s})\mathrm{d}s - \int_{0}^{\cdot} \phi^{\prime}(X_{s}) \mathrm{d}K_{s},
\end{equation}
and $\mathbb{P}$-a.s for every $t\in[0,T]$;
\item[ii)] $X_{t} \geq 0 $;
\item[iii)] $\displaystyle\int_{0}^{T} X_{t} \mathrm{d}K_{t} =0$.
\end{itemize}
\end{description}
\begin{remark}
Using the same arguments as in the proof of Proposition IV-2.1 in \cite{watanabe}, we can show that the existence of weak solutions to the system \eqref{eq01.1} is equivalent to the existence of solutions to the associated martingale problem $\mathbf{(6^{\prime})}$.
\end{remark}
Now, we introduce a relaxed control framework, where the controls are adapted processes taking values in the space of probability measures on $U$. Specifically, the $U$-valued process $u$ is replaced by an $\mathcal{P}(U)$-valued process $Q$.
\begin{definition} \label{relaxed} A tuple $r = (\Omega,\mathcal{F},\mathbb{F},\mathbb{P},Q,X,K)$ is called a relaxed control if 
\begin{description}
 \item[1)] $(\Omega,\mathcal{F},\mathbb{P})$ is a probability space equipped with a filtration $\mathbb{F}=(\mathcal{F}_t)_{t\geq 0}$;
\item[2)] $Q : [0,T] \times \Omega \rightarrow \mathcal{P}(U)$ is an $\mathbb{F}$-progressively measurable process, and $K$ is an $\mathbb{F}$-adapted process such that $K \in \mathcal{A}[0,T]$;
\item[3)] $X$ is an $\mathbb{F}$-adapted process, $X \in \mathcal{C}[0,T]$ and for each $\phi \in \mathcal{C}^{2}_{b}(\mathbb{R})$, $\mathcal{M}^{\phi}$ is $(\mathbb{F},\mathbb{P})$-martingale, where 
\begin{equation} \label{equ02.5}
\mathrm{}\mathcal{M}^{\phi}_{\cdot} := \phi(X_{\cdot}) - \int_{0}^{\cdot} \int_{U}^{} \mathcal{L} \phi(s,X_{s},u)Q_{s}(\mathrm{d}u)\mathrm{d}s - \int_{0}^{\cdot} \phi'(X_{s}) \mathrm{d}K_{s},
\end{equation}
and $\mathbb{P}$-a.s for every $t\in[0,T]$, the following holds true
\item[4)] $X_{t} \geq 0 $;
\item[5)] $\displaystyle\int_{0}^{T} X_{t} \mathrm{d}K_{t} =0$;
\item[6)] $\mathbb{P}\circ(X_0)^{-1}= \lambda$;
\end{description}
The collection of relaxed controls is denoted by $\tilde{\mathcal{R}}$.

\end{definition}
The cost functional corresponding to a relaxed control $r$ is defined by
\begin{equation}
\hat{J}(r) := \mathbb{E}^{\mathbb{P}} \left( \int_{0}^{T} \int_{U}^{} f(t,X_{t},u)Q_{t}(\mathrm{d}u)\mathrm{d}t + \int_{0}^{T} h(t, X_t)\mathrm{d}K_{t} + g(X_{T})\right).
\end{equation}

\begin{remark}\label{strict-control.}
\begin{itemize}
\item[i)] The relaxed control  $ (\Omega, \mathcal{F}, \mathbb{F}, \mathbb{P}, X, Q, K)$ is said to be strict if $Q_t =\mathrm{} \delta_{u_t}$ for some progressively measurable $U$-valued process $u$.
\item[ii)] Any strict control corresponds to a relaxed control. Indeed, If $\alpha =(\Omega, \mathcal{F}, \mathbb{F}, \mathbb{P}, X, u, K)  \in \Lambda$, then by considering $r=(\Omega, \mathcal{F}, \mathbb{F}, \mathbb{P}, X, Q, K))$ where $(Q_t(\mathrm{d}u) := \delta_{u_t}(\mathrm{d}u))$, it is straightforward to verify that $r$ is a relaxed control. Moreover, we have 
\begin{equation*}
\hat{J}(r) = J(\alpha),
\end{equation*}
 therefore,
\begin{equation*}
\inf_{r\in\tilde{\mathcal{R}} }\hat{J}(r) \leq \inf_{\alpha \in \Lambda} J(\alpha).
\end{equation*}
To obtain the reverse inequality, define, for each $ (t, x) \in [0, T] \times \mathbb{R} $, the set 
\begin{equation*}
S(t, x) := \{ (\sigma^{2}(t, x, u), b(t, x, u), e) : e \geq f(t, x, u), u \in U \}.
\end{equation*} 
If $ S(t, x)$ is convex for each $ (t, x) \in [0, T] \times \mathbb{R} $, it can be shown that for any relaxed control there exists a strict control with an equal or smaller cost. Indeed, since the cost associated to $K$ is independent of the control, by the reasoning outlined in the proof of Theorem 3.6 in \cite{Haussmann-lepeltier}, for any relaxed control $ r = (\Omega, \mathcal{F}, \mathbb{F}, P, X, Q, K) $, there exists a progressively measurable $U$-valued process $ \tilde{u}$ and an $ \mathbb{R}^+ $-valued process $ \tilde{v} $ such that, for almost all $ (t, \omega) \in [0, T] \times \Omega $,
\begin{align*}
&\left( \int_U \sigma^{2}(t, X_t(\omega), u) Q_t(\omega, \mathrm{d}u), \int_U b(t, X_t(\omega), u) Q_t(\omega, \mathrm{d}u), \int_U f(t, X_t(\omega), u) Q_t(\omega, \mathrm{d}u) \right) \\
&= \Big(\sigma^2(t, X_t(\omega), \tilde{u}_t(\omega)), b(t, X_t(\omega), \tilde{u}_t(\omega)), f(t, X_t(\omega), \tilde{u}_t(\omega)) + \tilde{v}_t(\omega) \Big).
\end{align*}
Then, $ \rho = (\Omega, \mathcal{F}, \mathbb{F}, P, X, \tilde{u}, K) $ is a strict control with an equal or smaller cost.
\end{itemize}
\end{remark}
\subsubsection{Control rules }

In the following, we denote by $\mathcal{C}^+$ the space of all continuous functions from $[0,T]$ to $\mathbb{R}^+$, and by $\mathcal{A}^+$ the subset of $\mathcal{C}^+$ consisting of all non-decreasing functions. Both spaces are  equipped with the topology of uniform convergence. We denote by $\mathcal{U}$ the set of all measures on $[0,T] \times U$ such that the first marginal is the Lebesgue measure on $[0,T]$ and the second marginal is a probability measure on $U$. 
The space $\mathcal{U}$ is equipped with the 2-Wasserstein metric, naturally adapted from \eqref{Wassersteindis}, as follows:
\begin{equation*}
d_{\mathcal{U}}(q_1, q_2) := W_{[0,T] \times U, 2}\left(\frac{q_1}{T}, \frac{q_2}{T}\right). 
\end{equation*}
Since $U$ is compact, this definition makes $\mathcal{U}$ a compact, complete, and  separable metric space.\\
We consider the canonical space 
\begin{equation*}
\Omega = \mathcal{C}^+ \times \mathcal{U} \times \mathcal{A}^+,
\end{equation*}
and the filtration $ \mathbb{F}:=(\mathcal{F}_{t})_{t\in [0,T]}$ is generated by the coordinate projections $X$, $Q$, $K$. Specifically, for each $\omega = (x,q,k) \in \Omega$,
\begin{equation*}
X(\omega) =x,\,\, Q(\omega) = q,\,\, K(\omega) = k,
\end{equation*}
and for $t \in [0,T]$, $\mathcal{F}_{t} := \mathcal{F}^{X}_{t} \otimes \mathcal{F}^{Q}_{t} \otimes \mathcal{F}^{K}_{t}$, where 
\begin{equation*}
\mathcal{F}^{X}_{t} = \sigma(X_{s}, s \leq t),\,\, \mathcal{F}^{Q}_{t} = \sigma(Q(F), F \in \mathcal{B}([0,t] \times U) ), \,\, \mathcal{F}^{K}_{t} = \sigma(K_{s}, s \leq t).
\end{equation*}
The following arguments demonstrate that relaxed  controls can be defined using projections mapping. Indeed, since $[0,T]$ and $U$ are compact, each $q \in \mathcal{U}$ can be disintegrated as the following
\begin{equation*}
q(\mathrm{d}u,\mathrm{d}t) = q_{t}(\mathrm{d}u) \mathrm{d}t ,
\end{equation*}
for some measurable $\mathcal{P}(U)$-valued function $(q_{t})_{t\in [0,T]}$.\\
From the definition of the space $\mathcal{U}$ and the Lemma 3.2 in \cite{laker}, there exists an $(\mathcal{F}_{t}^{Q})_{t \in [0, T]}$-predictable $\mathcal{P}(U)$-valued process $\Gamma$ such that for each $q \in \mathcal{U}$,
\begin{equation*}
\Gamma_{t}(q) = q_{t},\, \text{a.e}. \,\, t \in [0,T].
\end{equation*}
Thus, the process $Q^{0}_{t} := \Gamma_{t} \circ Q$ is $\mathbb{F}$-predictable. As a result, for each $ \omega = (x,q,k)$,
\begin{equation*}
Q(\omega)(\mathrm{d}u,\mathrm{d}t)=q(\mathrm{d}u,\mathrm{d}t) = q_{t}(\mathrm{d}u)\mathrm{d}t = \Gamma_{t}(q)(\mathrm{d}u)\mathrm{d}t = \Gamma_{t} \circ Q(\omega)(\mathrm{d}u)\mathrm{d}t = Q^{0}(\omega)(\mathrm{d}u)\mathrm{d}t.
\end{equation*}
This implies an adapted disintegration of Q in terms of the $\mathbb{F}$-progressively measurable process
\begin{equation*}
Q^{0} : [0, T] \times \Omega \longrightarrow \mathcal{P}(U)
\end{equation*}
and therefore enables us to define control rules. For now on, we let $(Q_{t})_{ t \in [0,T]} := (Q^{0}_{t})_{t \in [0,T]}$ for simplicity.
\begin{definition}
For the canonical path space $\Omega$, the canonical filtration $\mathbb{F}$, and the coordinate projections $(X,Q,K)$ introduced above, if $r = (\Omega, \mathcal{F}, \mathbb{F},\mathbb{P}, X, Q, K)$ is a relaxed control, then $\mathbb{P}$ is called a control rule. We denote by $\mathcal{R}$ the space of control rules. In this case, the associated cost functional is defined as 
\begin{equation*}
\tilde{J}(\mathbb{P}) := \hat{J}(r).
\end{equation*}
\begin{remark}
    It is evident that
\begin{equation*}
\inf_{ \mathbb{P} \in \mathcal{R}} \tilde{J}(\mathbb{P}) \geq \inf_{ r\in\tilde{\mathcal{R}} } \hat{J}(r).
\end{equation*}
Conversely, for any relaxed control $r$, one can construct a control rule $\mathbb{P} \in \mathcal{R}$ such that $ \tilde{J}(\mathbb{P}) = \hat{J}(r)$. The proof, similar to that of \cite[Proposition 2.6]{haussman}, applies here in the context of reflected SDEs. In simpler terms, the optimization problems involving relaxed controls and control rules are equivalent. Therefore, it suffices to consider control rules.
\end{remark}
\end{definition}
The next proposition represents $\mathcal{R}$ in terms of reflected SDEs driven by martingale measure.
For the notation of martingale measure and related stochastic integrals used in the result below, we refer to the article \cite{Mel}.\\
The  following proposition  is frequently used in the main text.
\begin{proposition} \label{rep mart}
The existence of a solution $\mathbb{P}$ to the martingale problem $\eqref{equ02.5}$ is equivalent to the existence of a weak solution to the following reflected SDEs:
\begin{equation} \label{eqrep}
\begin{cases} 
\mathrm{d}\tilde{X}_{t} = \displaystyle\int_{U}^{} b(t,\tilde{X}_{t},u)\tilde{Q}_{t}(\mathrm{d}u)\mathrm{d}t + \displaystyle\int_{U}^{} \sigma(t,\tilde{X}_{t},u)\tilde{M}(\mathrm{d}u,\mathrm{d}t) + \mathrm{d}\tilde{K}_{t}, \,\, & \tilde{X}_0  \sim \lambda; \\
\tilde{X}_{t} \geq 0 \,\, \text{a.s};\\
\displaystyle\int_{0}^{T}  \tilde{X}_{t} \mathrm{d}\tilde{K}_{t} = 0 \,\, \text{a.s},
\end{cases}
\end{equation}
where $\tilde{X}$, $\tilde{M}$ and $\tilde{K}$ are defined on some extension $(\tilde{\Omega},\tilde{\mathcal{F}}, \tilde{\mathbb{Q}})$ and $\tilde{M}$ is a martingale measure with intensity $\tilde{Q}$. Moreover, the two solutions are related by
\begin{equation*}
\mathbb{P}\circ (X, Q, K)^{-1} = \tilde{\mathbb{Q}}\circ (\tilde{X}, \tilde{Q}, \tilde{K})^{-1}.
\end{equation*}
\end{proposition}
\begin{proof}
Since $\mathbb{P}$ solves the martingale problem \eqref{equ02.5}, it follows from Proposition A.2 in \cite{Fu11} that there exists an extension of $(\Omega, \mathcal{F}, (\mathcal{F}_{t})_{t\in[0,T]},\mathbb{P})$, $(\tilde{\Omega}, \tilde{\mathcal{F}}, (\tilde{\mathcal{F}})_{t\in[0,T]}, \tilde{\mathbb{P}})$, on which a tuple of adapted stochastic processes $(\tilde{X}, \tilde{Q}, \tilde{K}, \tilde{M})$ is defined such that for all $t \in [0,T]$
\begin{equation*}
\mathrm{d}\tilde{X}_{t} = \int_{U}^{} b(t,\tilde{X}_{t},u)\tilde{Q}_{t}(\mathrm{d}u)\mathrm{d}t + \int_{U}^{} \sigma(t,\tilde{X}_{t},u)\tilde{M}(\mathrm{d}u,\mathrm{d}t) + \mathrm{d}\tilde{K}_{t},
\end{equation*}
and two tuples are related by $\mathbb{P}\circ (X, Q, K)^{-1} = \tilde{\mathbb{Q}}\circ (\tilde{X}, \tilde{Q}, \tilde{K})^{-1}$.\\
The processes $\tilde{X}$ and $\tilde{K}$ are non-negative due the fact that 
$(\tilde{\Omega}, \tilde{\mathcal{F}}, (\tilde{\mathcal{F}_{t}})_{t\in[0,T]}, \tilde{\mathbb{P}})$ is an extension of $(\Omega, \mathcal{F}, (\mathcal{F}_{t})_{t\in[0,T]},\mathbb{P})$. For the Skorokhod condition, we have 
\begin{equation*}
    \tilde{\mathbb{E}}\left(\int_{0}^{T} \tilde{X}_{t} \mathrm{d}\tilde{K}_{t}\right) = \mathbb{E}\left(\int_{0}^{T} X_{t} \mathrm{d}K_{t}\right)  =0,
\end{equation*}
since $\displaystyle\int_{0}^{T} \tilde{X}_{t} \mathrm{d}\tilde{K}_{t}$ is non-negative, we obtain $\displaystyle\int_{0}^{T} \tilde{X}_{t} \mathrm{d}\tilde{K}_{t} = 0$ $\tilde{\mathbb{P}}$-a.s.\\The inverse can be verified by applying Itô's formula.
\end{proof}
\subsection{Relaxed MFGs with controlled RSDE dynamics}\label{relaxed MFG}
We consider a mean field game associated with controlled RSDEs of the form \eqref{control-point} and \eqref{mean field-point}.

As a first step, we fix $\mu \in \mathcal{P}_{2}(\mathcal{C}^+)$ and study the optimal control problem of a representative agent, which consist in minimizing  
$$ \mathbb{E}\left(\int_{0}^{T} f(t,X_{t}, \mu_t, u_{t})\mathrm{d}t + \int_{0}^{T} h(t, X_t, \mu_t) \mathrm{d}K_{t}+g(X_T, \mu_T) \right), $$
over controls $u$, subject to:
\begin{equation*}
\begin{cases}
X_t =X_0+\displaystyle\int_0^t b(s, X_s, \mu_s, u_s)\mathrm{d}s + \displaystyle\int_0^t \sigma(s, X_s, \mu_s, u_s) \mathrm{d}B_s +  K_t, & t \in [0, T]; \\
X_0  \sim \lambda;\\
X_t \geq 0 \quad \text{a.s}; \\
\displaystyle\int_0^T X_t \mathrm{d}K_t = 0 \quad \text{a.s.},
\end{cases}
\end{equation*} 
\begin{definition}
Let $ \mu $ be a given probability measure in $ \mathcal{P}_2(\mathcal{C}^+) $. We say that $P$ is a control rule w.r.t $\mu$ if it satisfies the following conditions:
\begin{description}
\item[1)] $(\Omega,\mathcal{F},(\mathcal{F})_{t\in [0,T]}, P)$ is the canonical probability space and $(X, Q, K)$ are the coordinate projections.
\item[2)] For each $\phi \in \mathcal{C}^{2}_{b}(\mathbb{R})$, $\mathcal{M}^{\mu,\phi}$ is well defined $P$-continuous martingale, where 
\begin{equation} \label{equ2.5}
\mathcal{M}^{\mu,\phi}_{t} := \phi(X_{t}) - \int_{0}^{t} \int_{U}^{} \mathcal{L} \phi(s,X_{s},\mu_s,u)Q_{s}(\mathrm{d}u)\mathrm{d}s - \int_{0}^{t} \phi'(X_{s}) \mathrm{d}K_{s}.
\end{equation} 
\item[3)] $\displaystyle\int_{0}^{T} X_{t} \mathrm{d}K_{t} =0$  $P$.a.s.
\item[4)] $P\circ (X_0)^{-1} = \lambda.$
\end{description}
\end{definition}
For a fixed measure $\mu \in \mathcal{P}_{2}(\mathcal{C}^+)$, the corresponding set of control rules is denoted by $\mathcal{R}(\mu)$, the cost functional corresponding to a control rule $P \in \mathcal{R}(\mu) $ is
\begin{equation*}
J(\mu, P) := \mathbb{E}^{\mathbb{P}} \left( \int_{0}^{T} \int_{U}^{} f(t,X_{t},\mu_t, u)Q_{t}(\mathrm{d}u)\mathrm{d}t + \int_{0}^{T} h(t, X_t, \mu_t)\mathrm{d}K_{t} + g(X_{T}, \mu_T)\right),
\end{equation*}
and the (possibly empty) set of optimal control rules is denoted by
\begin{equation*}
\mathcal{R}^*(\mu) := \arg\min_{P \in \mathcal{R}(\mu)} J(\mu, P).
\end{equation*}
\begin{definition}[Relaxed and Markovian MFG solutions]\label{definition2}
A probability measure $ P $ is called a \emph{relaxed MFG solution} if
\begin{equation*}
P \in \mathcal{R}^*(P \circ X^{-1}),
\end{equation*}
that is, $ P $ constitutes a fixed point of the set-valued mapping
\begin{equation*}
\mathcal{P}_2(\mathcal{C}^+) \ni \mu \longmapsto 
\{\, P \circ X^{-1} : P \in \mathcal{R}^*(\mu) \,\}.
\end{equation*}

Moreover:
\begin{enumerate}
    \item If there exists a progressively measurable control process $ u $ such that
    \begin{equation*}
    P\big[\, Q(\mathrm{d}t, \mathrm{d}u) = \delta_{u_t}(\mathrm{d}u)\,\mathrm{d}t \,\big] = 1,
    \end{equation*}
    then $ P $ is referred to as a \emph{strict MFG solution}.

    \item If there exists a measurable function 
    $ \hat{q} : [0,T] \times \mathbb{R} \to \mathcal{P}(U) $ such that
    \begin{equation*}
    P\big[\, Q(\mathrm{d}t, \mathrm{d}u) = \hat{q}(t, X_t)(\mathrm{d}u)\,\mathrm{d}t \,\big] = 1,
    \end{equation*}
    then $ P $ is called a \emph{relaxed Markovian MFG solution}.

    \item If there exists a measurable function 
    $ \hat{\alpha} : [0,T] \times \mathbb{R} \to U $ such that
    \begin{equation*}
    P\big[\, Q(\mathrm{d}t, \mathrm{d}u) = \delta_{\hat{\alpha}(t, X_t)}(\mathrm{d}u)\,\mathrm{d}t \,\big] = 1,
    \end{equation*}
    then $ P $ is called a \emph{Markovian (strict) MFG solution}.
\end{enumerate}
\end{definition}
\begin{assumption}{\textbf{(A)}}
\label{ass.A}
We impose the following conditions on the coefficients $ b, \sigma, f, h, $ and $ g $.

\assumptionitem{A.1}{A.1} (\textbf{Regularity}) 
\label{A.1}
The functions $ b, \sigma, f, h, $ and $ g $ are measurable in $ t $ and continuous in $(x, \mu, u)$ on 
$[0,T] \times \mathbb{R} \times \mathcal{P}_2(\mathbb{R}^+) \times U$; in addition, $h$ is continuous in $t$.

\assumptionitem{A.2}{A.2}(\textbf{Lipschitz and Growth Conditions}) 
\label{A.2}
There exist constants $ C_1, C_2 > 0 $ such that for all 
$(t,u) \in [0,T] \times U$, $x,y \in \mathbb{R}$, and 
$\mu,\nu \in \mathcal{P}_2(\mathbb{R}^+)$,
\begin{equation*}
|b(t,x,\mu,u)-b(t,y,\nu,u)| + |\sigma(t,x,\mu,u)-\sigma(t,y,\nu,u)|
\le C_1\big(|x-y| + W_{\mathbb{R}^+,2}(\mu,\nu)\big),
\end{equation*}
and
\begin{equation*}
|\sigma(t,x,\mu,u)| + |b(t,x,\mu,u)|
\le C_2\left(1 + |x| + \Big(\int_{\mathbb{R}^+} |z|^2 \mu(\mathrm{d}z)\Big)^{1/2}\right).
\end{equation*}

\assumptionitem{A.3}{A.3} (\textbf{Dependence on the Measure}) 
\label{A.3}
The functions $\sigma^2, f, g, h$ are locally Lipschitz in the measure argument, uniformly in $(t,x,u)$.  
For each $\varphi \in \{\sigma^2, f, g, h\}$, there exists $C_3>0$ such that for all 
$(t,x,u) \in [0,T] \times \mathbb{R} \times U$ and 
$\mu_1,\mu_2 \in \mathcal{P}_2(\mathbb{R}^+)$,
\begin{equation*}
|\varphi(t,x,\mu_1,u) - \varphi(t,x,\mu_2,u)|
\le C_3\big(1 + F(W_{\mathbb{R}^+,2}(\mu_1,\delta_0), W_{\mathbb{R}^+,2}(\mu_2,\delta_0))\big)
W_{\mathbb{R}^+,2}(\mu_1,\mu_2),
\end{equation*}
where $F$ is locally bounded in its arguments.  
\assumptionitem{A.4}{A.4} (\textbf{Polynomial Growth}) 
\label{A.4}
There exists constant $ C_4 > 0 $ such that for all 
$(t,x,\mu,u) \in [0,T] \times \mathbb{R} \times \mathcal{P}_2(\mathbb{R}^+) \times U$,
\begin{equation*}
|g(x,\mu)| \le C_4 \left(1 + |x|^2 + \int_{\mathbb{R}^+} |z|^2 \mu(\mathrm{d}z) \right),
\end{equation*}
\begin{equation*}
|f(t,x,\mu,u)| \le C_4 \left(1 + |x|^2 + \int_{\mathbb{R}^+} |z|^2 \mu(\mathrm{d}z) \right),
\end{equation*}
\begin{equation*}
|h(t,x,\mu)| \le C_4\left(1 + |x| + \Big(\int_{\mathbb{R}^+} |z|^2 \mu(\mathrm{d}z)\Big)^{1/2}\right).
\end{equation*}

\assumptionitem{A.5}{A.5} (\textbf{Control Set}) 
\label{A.5}
The control space $ U $ is compact subset $U\subset\mathbb{R}^d$, $d\geq 1$.

\assumptionitem{A.6}{A.6} (\textbf{Initial Distribution})
\label{A.6}
The initial law $ \lambda $ admits a finite moment of order $ q' > 2 $, i.e.
$\lambda \in \mathcal{P}_{q'}(\mathbb{R}^+)$.
\end{assumption}

\begin{remark}
Assumption \hyperlink{A.3}{(A.3)} is introduced to guarantee the continuity of the cost functional $J$ and of the correspondence $\mathcal{R}$. In particular, when the diffusion coefficient $\sigma$ is bounded and Lipschitz continuous with respect to the measure argument, the associated function $\sigma^2$ automatically satisfies condition \hyperlink{A.3}{(A.3)}.
\end{remark}
\begin{assumption}{\textbf{(V)}}
\label{ass.V}
The diffusion coefficient $\sigma$ is uniformly elliptic: there exists a constant $\beta > 0$ such that
\begin{equation*}
\sigma^{2}(t, x, \mu, u) \geq \beta, 
\quad 
\forall (t, x, \mu, u) \in [0,T]\times\mathbb{R}\times\mathcal{P}_2(\mathbb{R}^+)\times U.
\end{equation*}
\end{assumption}
\begin{assumption}{\textbf{(C)}}
\label{ass.C}
For all $(t, x, \mu) \in [0,T] \times \mathbb{R} \times \mathcal{P}_2(\mathbb{R}^+)$, the subset
\begin{equation*}
\mathcal{S}(t,x,\mu)
:= 
\bigl\{
\bigl(b(t,x,\mu,\alpha),\ \sigma^2(t,x,\mu,\alpha),\ z\bigr)
:\ \alpha \in U,\ z \ge f(t,x,\mu,\alpha)
\bigr\}
\subset \mathbb{R} \times \mathbb{R} \times \mathbb{R}
\end{equation*}
is convex.
\end{assumption}
The following theorems summarize the principal contributions of this work by providing sufficient conditions for the existence of a relaxed solution to the mean field game, as well as for a strict Markovian MFG solution.
\begin{theorem} \label{theomfg}
Under Assumption $\ref{ass.A}$, there exists a relaxed MFG solution.   
\end{theorem}
\begin{theorem}\label{theorem:Markovian-new}
Suppose that Assumptions \ref{ass.A} and \ref{ass.V} hold. 
Then there exists a relaxed \emph{Markovian} MFG solution.  
Moreover, if Assumption \ref{ass.C} also holds, there exists a strict Markovian MFG solution.
\end{theorem}

\begin{proof}
Let $ P \in \mathcal{R}^*(\mu) $, with $ \mu = P \circ X^{-1} $, be a relaxed MFG solution provided by Theorem~\ref{theomfg}.  
We construct from $P$ a Markovian MFG solution that preserves the state marginals and cost, following the arguments of Corollary~3.8 in ~\cite{laker}.

By Theorem~2.5 in ~\cite{Mel}, there exists a measurable mapping
\begin{equation*}
\bar{\sigma} : [0,T] \times \mathbb{R} \times \mathcal{P}(\mathbb{R}) \times \mathcal{P}(U) \longrightarrow \mathbb{R},
\end{equation*}
continuous in $(x,\mu,q)$ for each $t$, such that
\begin{equation*}
\bar{\sigma}^{2}(t,x,\mu,q)
= \int_U \sigma^{2}(t,x,\mu,u)q(\mathrm{d}u)
\quad
\bar{\sigma}^{2}(t,x,\mu,\delta_u)
= \sigma^{2}(t,x,\mu,u).
\end{equation*}
Let $(\Omega,(\mathcal{F}_t),\mathbb{Q})$ be a filtered probability space supporting a process $(X,K)$ and a Wiener process $W$ such that
\begin{equation*}
\mathbb{Q}\circ(X,K,Q)^{-1}=P,
\end{equation*}
and where $(X,K)$ satisfies the reflected SDE
\begin{equation*}
\begin{cases}
\mathrm{d}X_t = \displaystyle \int_U b(t,X_t,\mu_t,u)\,Q_t(du)\,\mathrm{d}t
+ \bar{\sigma}(t,X_t,\mu_t,Q_t)\,\mathrm{d}W_t + \mathrm{d}K_t, \\[4pt]
\mathbb{Q}\circ(X_0 )^{-1}= \lambda, \\
X_t \geq 0,  \quad
\displaystyle\int_0^T X_t \,\mathrm{d}K_t = 0.
\end{cases}
\end{equation*}

Let $X^n$ be the solution of the penalized SDE
\begin{equation*}
\mathrm{d}X_t^n = \int_U b(t,X_t^n,\mu_t,u)\,Q_t(du)\,\mathrm{d}t + n(X_t^n)_{-}\,\mathrm{d}t
+ \bar{\sigma}(t,X_t^n,\mu_t,Q_t)\,\mathrm{d}W_t,\quad
X^n_0= X_0,
\end{equation*}
where $(x)_{-} = \max(0,-x)$ and $ K_t^n := \displaystyle\int_0^t n(X_s^n)_{-}\,\mathrm{d}s$.  
By adapting the arguments of Theorem~4.1 in Słomiński~\cite{Slominski-measurable}, one obtains that
\begin{equation}\label{conv1.1}
P^n := \mathbb{Q}\circ(X^n,K^n,Q)^{-1} \longrightarrow P
\quad\text{in }\mathcal{P}_2(\Omega).
\end{equation}
As in Corollary~3.8 of~\cite{laker}, there exists a measurable function $\hat{q} : [0,T]\times\mathbb{R} \to \mathcal{P}(U)$ such that
\begin{equation*}
\hat{q}(t,X_t^n) = \mathbb{E}^{\mathbb{Q}}[Q_t \mid X_t^n],
\quad \mathrm{d}t\times\mathrm{d}\mathbb{Q}\text{-a.s.}
\end{equation*}
It then follows that, $\mathrm{d}t \times \mathrm{d}\mathbb{Q}$-a.e.,
\begin{equation*}
\int_U b(t,X^n_t,\mu_t,u)\, \hat{q}(t,X^n_t)(\mathrm{d}u)
+ n(\hat{X}^n_t)_{-}
= \mathbb{E}^{\mathbb{Q}}\!\left[
\int_U b(t,X^n_t,\mu_t,u)\, Q_t(\mathrm{d}u)
+ n(\hat{X}^n_t)_{-}
\,\middle|\, X^n_t
\right]
\end{equation*}
and
\begin{equation*}
\bar{\sigma}^{2}(t,X^n_t,\mu_t,\hat{q}(t,X^n_t))
=
\mathbb{E}^{\mathbb{Q}}\!\left[
\bar{\sigma}^{2}(t,X^n_t,\mu_t,Q_t)
\,\middle|\, X^n_t
\right].
\end{equation*} 
By the mimicking result Corollary 3.7 in \cite{brunick}, there exists a filtered probability space 
$(\hat{\Omega}^n,(\hat{\mathcal{F}}^n_t),\hat{\mathbb{Q}}^n)$ 
supporting a Wiener process $\hat{W}^n$ and an adapted process $\hat{X}^n$ satisfying
\begin{equation*}
\mathrm{d}\hat{X}^n_t
= b^1(t,\hat{X}^n_t)\mathrm{d}t
+ n(\hat{X}^n_t)_{-}\,\mathrm{d}t
+ \bar{\sigma}^1(t,\hat{X}^n_t)\mathrm{d}\hat{W}^n_t,
\end{equation*}
such that
\begin{equation}\label{equality1}
\hat{\mathbb{Q}}^n\circ(\hat{X}^n_t)^{-1}
= \mathbb{Q}\circ(X^n_t)^{-1}, \quad \forall t,
\end{equation}
where
\begin{equation*}
b^1(t,x) := \int_U b(t,x,\mu_t,u)\,\hat{q}(t,x)(du),
\quad
\bar{\sigma}^1(t,x) := \bar{\sigma}(t,x,\mu_t,\hat{q}(t,x)).
\end{equation*}
Note that the coefficients $(b^{1}, \bar{\sigma}^{1})$ are merely measurable, possibly discontinuous, and satisfy Condition~1.3 in \cite{Slominski-measurable}. By Theorem~2.1 of \cite{Slominski-measurable}, the sequence $(\hat{X}^{n}, \hat{K}^{n})$ is tight in $\mathcal{C}([0,T], \mathbb{R}^2)$. Moreover, by Assumption~\ref{ass.V}, the coefficient $\bar{\sigma}^{1}$ satisfies the uniform ellipticity condition.\\
Using the same arguments as in the proof of Theorem~3.1 of \cite{Slominski-measurable}, it follows that
\begin{equation}\label{convergence-law-mimicking}
(\hat{X}^n, \hat{K}^n) \to (\hat{X}, \hat{K}) \text{ in law,}
\end{equation}
where 
$(\hat{\Omega},\hat{\mathcal{F}},\hat{\mathbb{F}},\hat{\mathbb{Q}},\hat{W},\hat{X},\hat{K})$ 
is a weak solution of the reflected SDE on $D=(0,\infty)$,
\begin{equation*}
\begin{cases}
\mathrm{d}\hat{X}_t = b^1(t,\hat{X}_t)\,\mathrm{d}t
+ \bar{\sigma}^1(t,\hat{X}_t)\,\mathrm{d}\hat{W}_t + \mathrm{d}\hat{K}_t, \\[4pt]
\hat{X}_t \ge 0, \qquad
\displaystyle \int_0^T \hat{X}_s\,\mathrm{d}\hat{K}_s = 0.
\end{cases}
\end{equation*}

By Itô’s formula, for every $q \ge 1$ such that $2q \in [2,q']$, we obtain 
\begin{equation}\label{Dom-measurable}
\sup_{n} \mathbb{E}^{\mathbb{Q}^n}\left(
\sup_{0 \le t \le T} |\hat{X}^n_t|^{2q}
+ |\hat{K}^n_T|^{2q}
\right) < \infty.
\end{equation}
Another application of Itô’s formula yields
\begin{equation}\label{first-prop}
\hat{\mathbb{P}}
:= \hat{\mathbb{Q}} \circ \big(\hat{X}, \hat{K},
\hat{q}(t,\hat{X}_t)(\mathrm{d}u)\,\mathrm{d}t\big)^{-1}
\in \mathcal{R}(\mu).
\end{equation}
Moreover, combining \eqref{convergence-law-mimicking}, \eqref{equality1}, and \eqref{conv1.1}, for all $t \in [0, T]$, we obtain  
\begin{align} \label{second-prop}
\begin{split}
 \hat{\mathbb{Q}}\circ(\hat{X}_t)^{-1}
&= \lim_{n \to \infty} \hat{\mathbb{Q}}^n\circ(\hat{X}^n_t)^{-1} \\
&= \lim_{n \to \infty} \mathbb{Q}\circ(X^n_t)^{-1}\\
&= \mathbb{Q}\circ(X_t)^{-1}
= P\circ(X_t)^{-1}
= \mu_t.
\end{split}
\end{align}
The Skorokhod's representation theorem, implies that there exists $(\tilde{X}^n, \tilde{K}^n)$ and $ (\tilde{X}, \tilde{K})$ defined on some probability space $(\tilde{\Omega}, \tilde{\mathcal{F}}, \tilde{\mathbb{Q}})$, such that \begin{equation}\label{samelaw2} Law(\tilde{X}^n, \tilde{K}^n) = Law(\hat{X}^n, \hat{K}^n), \quad Law(\tilde{X}, \tilde{K}) = Law(\hat{X}, \hat{K}) \end{equation} 
and \begin{equation}\label{conv2} (\tilde{X}^n, \tilde{K}^n) \to (\tilde{X}, \tilde{K}) \quad \text{$\tilde{\mathbb{Q}}$-a.s.} 
\end{equation} 
Hence, we have
\begin{align} \label{third-prop}
\begin{split}
J(\mu, \hat{\mathbb{P}}) &= \mathbb{E}^{\hat{\mathbb{Q}}} \left[ g(\hat{X}_T, \mu_T) + \int_0^T  \int_U \hat{q}(t, \hat{X}_t)(\mathrm{d}u)\, f(t, \hat{X}_t, \mu_t, u)\mathrm{d}t  + \int_{0}^{T} h(t, \hat{X}_t, \mu )\,\mathrm{d}\hat{K}_t \right]\\
&=\mathbb{E}^{\tilde{\mathbb{Q}}} \left[ g(\tilde{X}_T, \mu_T) + \int_0^T \int_U \hat{q}(t, \tilde{X}_t)(\mathrm{d}u)\, f(t, \tilde{X}_t, \mu_t, u) \mathrm{d}t  + \int_{0}^{T} h(t, \tilde{X}_t, \mu )\,\mathrm{d}\tilde{K}_t \right] \\ &= \lim_{n \to \infty} \mathbb{E}^{\tilde{\mathbb{Q}}} \left[ g(\tilde{X}^n_T, \mu_T) + \int_0^T \int_U \hat{q}(t, \tilde{X}^n_t)(\mathrm{d}u)\, f(t, \tilde{X}^n_t, \mu_t, u) \mathrm{d}t  + \int_{0}^{T} h(t, \tilde{X}^n_t, \mu_t )\,\mathrm{d}\tilde{K}^n_t \right] \\ &= \lim_{n \to \infty} \mathbb{E}^{\mathbb{Q}^n} \left[ g(\hat{X}^n_T, \mu_T) + \int_0^T \int_U \hat{q}(t, \hat{X}^n_t)(\mathrm{d}u)\, f(t, \hat{X}^n_t, \mu_t, u) \mathrm{d}t  + \int_{0}^{T} h(t, \hat{X}^n_t, \mu_t )\,\mathrm{d}\hat{K}^n_t \right] \\ &= \lim_{n \to \infty} \mathbb{E}^{\mathbb{Q}} \left[ g(X^n_T, \mu_T) + \int_0^T \int_U \hat{q}(t, X^n_t)(\mathrm{d}u)\, f(t, X^n_t, \mu_t, u) \mathrm{d}t  + \int_{0}^{T} h(t, X^n_t, \mu_t )\,\mathrm{d}K^n_t \right] \\ &= \lim_{n \to \infty} \mathbb{E}^{\mathbb{Q}} \left[ g(X^n_T, \mu_T) + \int_0^T  \int_U Q_t(\mathrm{d}u)\, f(t, X^n_t, \mu_t, u) \mathrm{d}t  + \int_{0}^{T} h(t, X^n_t, \mu_t )\,\mathrm{d}K^n_t \right] \\ &= \mathbb{E}^{\mathbb{Q}} \left[ g(X_T, \mu_T) + \int_0^T \int_U Q_t(\mathrm{d}u)\, f(t, X_t, \mu_t, u) \mathrm{d}t  + \int_{0}^{T} h(t,X_t, \mu_t )\,\mathrm{d}K_t \right] \\ &= J(\mu, P).
\end{split}
\end{align} 
For the first convergence, we apply \textit{Krylov’s estimate} to the measurable mapping  
\begin{equation*}
x \mapsto \int_{U} f\big(t, x, \mu_t, u\big)\, \hat{q}(t, x)(\mathrm{d}u),
\end{equation*}
in conjunction with the continuity of the functional  
\begin{equation*}
(x, k) \mapsto \int_{0}^{T} h(t, x_t, \mu_t)\,\mathrm{d}k_t,
\end{equation*}
and condition~\eqref{Dom-measurable}. 
Furthermore, the uniform growth bound follows from assumption \hyperlink{A.4}{(A.4)} 
\begin{equation*}
\left| 
\int_{0}^{T}\!\!\int_{U} f\big(t, x, \mu_t, u\big)\, \hat{q}(t, x)(\mathrm{d}u)
+ \int_{0}^{T} h(t, x_t, \mu_t)\,\mathrm{d}k_t 
\right|
\le C\big(1+\|x\|_T^2+\|k\|_T^2\big)
\end{equation*}
guarantees the necessary integrability, thereby justifying the passage to the limit via the dominated convergence theorem.
The final convergence follows from~\eqref{conv1.1}.

Therefore, by combining \eqref{first-prop}, \eqref{second-prop}, and \eqref{third-prop}, we conclude that $\hat{\mathbb{P}}$ constitutes a relaxed Markovian MFG solution.

Finally, under Assumption \ref{ass.C}, measurable selection arguments (as in Theorem~3.7 of Lacker~\cite{laker}) provide measurable functions
\begin{equation*}
\hat{\alpha} : [0,T]\times\mathbb{R} \to U,
\qquad
\hat{v} : [0,T]\times\mathbb{R} \to \mathbb{R}^+,
\end{equation*}
satisfying, for every $t$ and almost every $\omega$,
\begin{equation*}
\int_U \hat{q}(t,\hat{X}_t(\omega))(du)\,(b,\sigma^{2},f)(t,\hat{X}_t(\omega),\mu_t,u)
=
(b,\sigma^{2},f)(t,\hat{X}_t(\omega),\mu_t,\hat{\alpha}(t,\hat{X}_t(\omega)))
+
(0,0,\hat{v}(t,\hat{X}_t(\omega))).
\end{equation*}
The relaxed control can thus be replaced by the strict control $u_t = \hat{\alpha}(t,\hat{X}_t)$ without increasing the cost.  
Defining
\begin{equation*}
\hat{\mathbb{P}}'
:= \hat{\mathbb{P}}\circ(\hat{X},\hat{K},\delta_{\hat{\alpha}(t,\hat{X}_t)}(\mathrm{d}u)\mathrm{d}t)^{-1},
\end{equation*}
we obtain that $\hat{\mathbb{P}}'$ is a strict Markovian MFG solution.
\end{proof}
\begin{remark}
Assumptions \ref{ass.A} and \ref{ass.C}, together with the reasoning in point (ii) of Remark \ref{strict-control.}, imply the existence of a strict MFG solution.
\end{remark}
\section{Proof of Theorem \ref{theomfg}}\label{proof of theorem-mfg}
According to Definition~\ref{definition2}, a probability distribution 
$ P \in \mathcal{P}_2(\Omega) $ is said to be a relaxed MFG solution if it constitutes a fixed point of the correspondence 
\begin{equation*}
\Phi : \mathcal{P}_2(\mathcal{C}^+) \rightrightarrows \mathcal{P}_2(\mathcal{C}^+), 
\qquad 
\mu \mapsto \Phi(\mu) := \{ P \circ X^{-1} : P \in \mathcal{R}^*(\mu) \}.
\end{equation*}
Hence, the proof of Theorem~\ref{theomfg} reduces to establishing the existence 
of a fixed point of $ \Phi $. To this end, we apply Kakutani’s fixed-point theorem.  

The verification of Kakutani’s assumptions relies on Berge’s maximum theorem, which requires that the correspondence $ \mathcal{R} $ be 
continuous (as shown in Lemmas~\ref{lemma-lower} and~\ref{lemma-upper}), compact-valued, and that the cost functional $ J $ be continuous on the graph of $ \mathcal{R} $,
\begin{equation*}
\mathrm{Gr}(\mathcal{R}) := 
\{ (\mu, P) \in \mathcal{P}_2(\mathcal{C}^+) \times \mathcal{P}_2(\Omega) 
: P \in \mathcal{R}(\mu) \},
\end{equation*}
as established in Lemma~\ref{contcost}.  

By Definition~\ref{def of set-valued conti}, the correspondence $ \mathcal{R} $
is continuous if and only if it is both \emph{lower} and \emph{upper hemicontinuous}. 
These two properties are verified in Lemmas~\ref{lemma-lower} and~\ref{lemma-upper}, respectively.

\begin{lemma}\label{lemma-lower}
$\mathcal{R}$ is lower hemicontinuous. 
\end{lemma}
\begin{proof}
Let $(\mu^n, \mu)$ be a sequence in $\mathcal{P}_2(\mathcal{C}^+) \times \mathcal{P}_2(\mathcal{C}^+)$ such that $\mu^n \to \mu$, and let $P \in \mathcal{R}(\mu)$.  
To prove that the correspondence $\mathcal{R}$ is lower hemicontinuous, it suffices by point~(i) of Definition~\ref{def of set-valued conti} to show that there exists a sequence $P^n \in \mathcal{R}(\mu^n)$ satisfying $P^n \to P$ as $n \to \infty$.

By Proposition \ref{rep mart}, there exists a filtered probability space $ (\tilde{\Omega}, \tilde{\mathcal{F}}, \tilde{\mathcal{F}}_t, \mathbb{Q}) $ on which stochastic processes $(\tilde{X}, \tilde{K}, \tilde{Q}, \tilde{M} )$ are defined such that $(\tilde{X}, \tilde{K}) $ are $\tilde{\mathcal{F}}_t$-adapted, $ \tilde{M} $ $\tilde{\mathcal{F}}_t $-martingale measures on $ U \times [0, T] $ with intensity $\tilde{Q} $, $ P=P \circ (X, Q, K)^{-1} = \mathbb{Q}\circ (\tilde{X}, \tilde{Q}, \tilde{K})^{-1} $, and the state equation \eqref{eqrep} holds on $ \tilde{\Omega} $.

Let $ (\tilde{X}^n, \tilde{K}^n)$ be the strong solution to the following reflected SDE, whose existence follows from the Lipchitz condition in Assumption \hyperlink{A.2}{(A.2)} (see \cite{ elkarouiexistence}).
\begin{equation}
\begin{cases} 
\tilde{X}^n_{t} =\tilde{X}_{0}+ \displaystyle \int_{0}^{t}\int_{U}^{} b(s,\tilde{X}^n_{s}, \mu^n_s, u)\tilde{Q}_{s}(\mathrm{d}u)\mathrm{d}s +\int_{0}^{t} \int_{U}^{} \sigma(s,\tilde{X}^n_{s},\mu^n_s, u)\tilde{M}(\mathrm{d}u, \mathrm{d}s) + \tilde{K}^n_{t}; \\
\tilde{X}^n_{t} \geq 0 \,\, \text{a.s};\\
\displaystyle\int_{0}^{T}  \tilde{X}^n_{t} \mathrm{d}\tilde{K}^n_{t} = 0 \,\, \text{a.s}.
\end{cases}
\end{equation}

Applying Itô’s formula to the continuous semi-martingale $|\tilde{X}^n - \tilde{X}|^2$ yields to : 
\begin{align*}
\mathrm{d}|\tilde{X}^n_t - \tilde{X}_t|^2 &= 2\int_{U}^{}\big(\tilde{X}^n_t - \tilde{X}_t\big)\Big(b(t,\tilde{X}^n_t, \mu^n_t, u)-b(t,\tilde{X}_t, \mu_t, u)\Big)\tilde{Q}_t(\mathrm{d}u)\mathrm{d}t \\ & +2 \int_{U}^{}\big(\tilde{X}^n_t - \tilde{X}_t\big)\Big(\sigma(t,\tilde{X}^n_t, \mu^n_t, u)-\sigma(t,\tilde{X}_t, \mu_t, u)\Big)\tilde{M}(\mathrm{d}t, \mathrm{d}u) \\ & + \int_{U}^{}\big|\sigma(t,\tilde{X}^n_t, \mu^n_t, u)-\sigma(t,\tilde{X}_t, \mu_t, u)\big|^2\tilde{Q}_t(\mathrm{d}u) \mathrm{d}t\\ 
& + 2 (\tilde{X}^n_t - \tilde{X}_t)\mathrm{d}\tilde{K}^n_t -2 (\tilde{X}^n_t - \tilde{X}_t)\mathrm{d}\tilde{K}_t.
\end{align*} 
The Skorokhod conditions imply that $ (\tilde{X}^n_t - \tilde{X}_t) \,\mathrm{d}\tilde{K}^n_t \leq 0 $ and $ (\tilde{X}^n_t - \tilde{X}_t) \,\mathrm{d}\tilde{K}_t \geq 0 $. Then, by taking the supremum over $[0, T]$ and the expectation, we obtain:
\begin{equation}
\begin{split}
    \mathbb{E}^{\mathbb{Q}}\left(\sup_{t \in [0, T]}\big|\tilde{X}^n_t - \tilde{X}_t\big|^2\right) &\leq 2\mathbb{E}^{\mathbb{Q}}\left(\int_{0}^{T}\int_{U}^{}\big|(\tilde{X}^n_t - \tilde{X}_t)(b(t,\tilde{X}^n_t, \mu^n_t, u)-b(t,\tilde{X}_t, \mu_t, u))\big|\tilde{Q}_t(\mathrm{d}u)\mathrm{d}t \right) \\  +2\mathbb{E}^{\mathbb{Q}}&\left (\sup_{s \in [0, T]}\Big|\int_0^s \int_{U}^{}(\tilde{X}^n_t - \tilde{X}_t)(\sigma(t,\tilde{X}^n_t, \mu^n_t, u)-\sigma(t,\tilde{X}_t, \mu_t, u))\tilde{M}(\mathrm{d}t, \mathrm{d}u)\Big|\right) \\  +\mathbb{E}^{\mathbb{Q}}\Big(&\int_{0}^{T}\int_{U}^{}\big|\sigma(t,\tilde{X}^n_t, \mu^n_t, u)-\sigma(t,\tilde{X}_t, \mu_t, u)\big|^2\tilde{Q}_t(\mathrm{d}u)\mathrm{d}t\Big) \label{suptherm}.
    \end{split}
\end{equation}
By applying the Burkholder-Davis-Gundy inequality, we obtain a constant $ C > 0$, independent of $ n $, such that 
\begin{equation}
\begin{split}
&\mathbb{E}^{\mathbb{Q}}\left (\sup_{s \in [0, T]}\Big|\int_0^s \int_{U}^{}(\tilde{X}^n_t - \tilde{X}_t)(\sigma(t,\tilde{X}^n_t, \mu^n_t, u)-\sigma(t,\tilde{X}_t, \mu_t, u))\tilde{M}(\mathrm{d}t, \mathrm{d}u)\Big| \right ) \\
&\leq C \mathbb{E}^{\mathbb{Q}}\left ( \int_{U}^{} \int_{0}^{T}\big|\tilde{X}^n_t -\tilde{X}_t\big|^2 \big|\sigma(t,\tilde{X}^n_t, \mu^n_t, u)-\sigma(t,\tilde{X}_t, \mu_t, u)\big|^2\tilde{Q}_t(\mathrm{d}u)\mathrm{d}t \right)^{1/2}\\
&\leq C\epsilon \mathbb{E}^{\mathbb{Q}}\left (\sup_{t \in [0, T]}\big|\tilde{X}^n_t -\tilde{X}_t\big|^2 \right )+\dfrac{C}{\epsilon} \mathbb{E}^{\mathbb{Q}}\left( \int_{0}^{T}\int_{U}^{}\big|\sigma(t,\tilde{X}^n_t, \mu^n_t, u)-\sigma(t,\tilde{X}_t, \mu_t, u)\big|^2\tilde{Q}_t(\mathrm{d}u)\mathrm{d}t\right) \label{marttherm}
\end{split}
\end{equation}
By plugging \eqref{marttherm} into \eqref{suptherm}, selecting a suitable $\epsilon$ and relying on the assumption \hyperlink{A.2}{(A.2)}, we deduce  
\begin{align*}
\mathbb{E}^{\mathbb{Q}}\left(\sup_{t \in [0, T]}\big|\tilde{X}^n_t -\tilde{X}_t\big|^2 \right)  &\leq C\mathbb{E}^{\mathbb{Q}}\left(\int_0^{T} \int_{U}^{}\Big| (\tilde{X}^n_s - \tilde{X}_{s})(b(s,\tilde{X}^n_s, \mu^n_s, u)-b(s,\tilde{X}_s, \mu_s, u))\Big|\tilde{Q}_s(\mathrm{d}u)\mathrm{d}s\right) \\
&+C\mathbb{E}^{\mathbb{Q}}\left(\int_{0}^{T} \int_{U}^{} \bigl|\sigma(s,\tilde{X}^n_s, \mu^n_s, u)-\sigma(s,\tilde{X}_s, \mu_s, u)\big|^2\tilde{Q}_s(\mathrm{d}u)\mathrm{d}s\right) \\
&\leq C\mathbb{E}^{\mathbb{Q}}\left(\int_0^{T}\int_U^{} (\big|\tilde{X}^n_s - \tilde{X}_{s}\big|^2+\big|b(s,\tilde{X}^n_s, \mu^n_s, u)-b(s,\tilde{X}_s, \mu_s, u)\big|^2)\tilde{Q}_s(\mathrm{d}u)\mathrm{d}s\right) \\
&+C\mathbb{E}^{\mathbb{Q}}\left(\int_{0}^{T} \int_{U}^{} \big|\sigma(s,\tilde{X}^n_s, \mu^n_s, u)-\sigma(s,\tilde{X}_s, \mu_s, u)\big|^2\tilde{Q}_s(\mathrm{d}u)\mathrm{d}s\right) \\
&\leq C \mathbb{E}^{\mathbb{Q}}\left(\int_0^{T} \big|\tilde{X}^n_s - \tilde{X}_{s}\big|^2\mathrm{d}s\right) + C\int_0^T W_2(\mu^n_s, \mu_s)^2\mathrm{d}s.
\end{align*}
From Gronwall’s inequality and Lemma \ref{sumconv}, we obtain
\begin{equation}\label{lower1}
\mathbb{E}^{\mathbb{Q}} \left( \sup_{t \in [0, T]} |\tilde{X}^n_t - \tilde{X}_t |^2\right) \to 0 .
\end{equation}
Similarly, by using the fact that 
\begin{multline*}
\tilde{K}^n_t-\tilde{K}_t=\int_0^t \int_{U}^{} (b(s,\tilde{X}^n_{s},\mu^n_s, u)-b(s,\tilde{X}_{s},\mu_s, u))\tilde{Q}_{s}(\mathrm{d}u)\mathrm{d}s\\+\int_0^t \int_{U}^{}(\sigma(s,\tilde{X}^n_{s},\mu^n_s, u)-\sigma(s,\tilde{X}_{s},\mu_s, u))\tilde{M}(\mathrm{d}u,\mathrm{d}s)
\end{multline*}
we can show that
\begin{equation} \label{lower2}
\mathbb{E}^{\mathbb{Q}} \left(\sup_{t \in [0, T]}|\tilde{K}^n_t - \tilde{K}_t|^2\right) \to 0 .
\end{equation}
Define $ P^n := \mathbb{Q} \circ (\tilde{Q}, \tilde{X}^n, \tilde{K}^n)^{-1} $, \eqref{lower1} and \eqref{lower2} implies that $P^n \to P$ in $\mathcal{P}_2(\Omega)$ and by using Itô's formula, we can verify that $ P_n \in \mathcal{R}(\mu_n) $. 
\end{proof}

\begin{lemma} \label{lemma-upper}
The correspondence $\mathcal{R}$ is upper hemicontinuous, and for every $\mu \in \mathcal{P}_{2}(\mathcal{C}^+)$, the set $\mathcal{R}(\mu)$ is compact.
\end{lemma}
\begin{proof}
\textbf{Step 1:} We first show that for each $\mu \in \mathcal{P}_{2}(\mathcal{C}^+)$, the set $\mathcal{R}(\mu)$ is closed.  

Let $(P^n)_{n \ge 1}$ be a sequence in $\mathcal{R}(\mu)$ such that $P^n \to P$ in $\mathcal{P}_2(\Omega)$.  
Define $\Psi(x,k) := \int_0^T x_t\, \mathrm{d}k_t$, which is continuous under uniform convergence and satisfies
\begin{equation*}
|\Psi(x,k)| \le C\!\left(1 + \|x\|_T^2 + \|k\|_T^2\right).
\end{equation*}
By (iv) of Theorem~\ref{convlaw},
\begin{equation*}
0 = \mathbb{E}^{P^n}\!\left(\int_0^T X_t\, \mathrm{d}K_t\right) 
\longrightarrow 
\mathbb{E}^{P}\!\left(\int_0^T X_t\, \mathrm{d}K_t\right),
\end{equation*}
and since $\int_0^T x_t\, \mathrm{d}k_t \ge 0$ for all $(x,k)\in \mathcal{C}^+\times\mathcal{A}^+$, we deduce
\begin{equation*}
\int_0^T X_t\, \mathrm{d}K_t = 0 \quad P\text{-a.s.}
\end{equation*}
The initial condition is also preserved:
\begin{equation*}
P \circ (X_0)^{-1} = \lim_{n \to \infty} P^n \circ (X_0)^{-1} = \lambda.
\end{equation*}

Let $0 \le s \le t \le T$, $\psi \in \mathcal{C}_0^{\infty}$, and let $H$ be a bounded $\mathcal{F}_s$-measurable random variable. Define
\begin{equation*}
\mathcal{M}_t^{\mu,\psi} 
:= \psi(X_t) 
- \int_0^t \mathcal{L}\psi(s,X_s,\mu_s,u)\, Q_s(\mathrm{d}u)\, \mathrm{d}s 
- \int_0^t \psi'(X_s)\, \mathrm{d}K_s,
\end{equation*}
with
\begin{equation*}
\mathcal{L}\psi(s,x,\mu,u)
= \tfrac{1}{2}\sigma^2(s,x,\mu,u)\psi''(x) + b(s,x,\mu,u)\psi'(x),
\end{equation*}
by the second point of Corollary~A.5 in \cite{laker}, the map 
\begin{equation*}
(q,x,k) \mapsto \psi(x_t) - \int_0^t\! \mathcal{L}\psi(s,x_s,\mu_s,u)\, q_s(\mathrm{d}u)\, \mathrm{d}s - \int_0^t \psi'(x_s)\, \mathrm{d}k_s
\end{equation*}
is continuous and quadratically bounded. Hence, by (iv) of Theorem~\ref{convlaw}, we obtain
\begin{equation*}
0 = \lim_{n \to \infty} 
\mathbb{E}^{P^n}\!\big[(\mathcal{M}_t^{\mu,\psi} - \mathcal{M}_s^{\mu,\psi}) H\big]
= 
\mathbb{E}^{P}\!\big[(\mathcal{M}_t^{\mu,\psi} - \mathcal{M}_s^{\mu,\psi}) H\big].
\end{equation*}
Thus $P$ satisfies the martingale property of $\mathcal{R}(\mu)$, and we conclude that $P \in \mathcal{R}(\mu)$.

\textbf{Step 2:} Next, let $\mu^n \to \mu$ in $\mathcal{P}_{2}(\mathcal{C}^+)$, and suppose that $P^n \in \mathcal{R}(\mu^n)$ for each $n$. 
Since $\mathcal{R}(\mu)$ is closed for every $\mu$, it follows from point~(ii) of Definition~\ref{def of set-valued conti} that the correspondence $\mathcal{R}$ is upper hemicontinuous whenever the sequence $(P^n)$ admits a limit point $P \in \mathcal{R}(\mu)$.

By Proposition~\ref{rep mart}, there exists a filtered probability space 
$(\Omega^n, \mathbb{F}^n = (\mathcal{F}_t^n), \mathbb{P}^n)$ supporting 
$\mathbb{F}^n$-adapted processes $(X^n, K^n)$ and an $\mathbb{F}^n$-martingale measure 
$M^n$ on $U \times [0,T]$ with intensity $Q^n$, such that 
$P^n = \mathbb{P}^n \circ (X^n, Q^n, K^n)^{-1}$ and
\begin{equation}\label{tight1}
\begin{cases}
\mathrm{d}X^n_t = \displaystyle\int_U b(t, X^n_t, \mu^n_t, u) Q^n_t(\mathrm{d}u)\, \mathrm{d}t 
+ \displaystyle\int_U \sigma(t, X^n_t, \mu^n_t, u) M^n(\mathrm{d}u, \mathrm{d}t)
+ \mathrm{d}K^n_t, & t \in [0,T], \\[0.4em]
X^n_0 \sim \lambda, \quad X^n_t \ge 0 \text{ a.s.}, \quad
\displaystyle\int_0^T X^n_t\, \mathrm{d}K^n_t = 0 \text{ a.s.}
\end{cases}
\end{equation}

We first show that $(P^n)$ is relatively compact in $\mathcal{P}_{2}(\Omega)$.  
By Lemma~A.2 in \cite{laker}, it suffices to prove that 
$\{\mathbb{P}^n \circ (X^n)^{-1}\}_{n \ge 1}$, 
$\{\mathbb{P}^n \circ (Q^n)^{-1}\}_{n \ge 1}$, and 
$\{\mathbb{P}^n \circ (K^n)^{-1}\}_{n \ge 1}$ 
are relatively compact.  

Since $\mathcal{U}$ is compact, 
$\{\mathbb{P}^n \circ (Q^n)^{-1}\}_{n \ge 1}$ 
is tight, so the sequence is relatively compact under the 2-Wasserstein topology.  
It remains to establish compactness for 
$\{ \mathbb{P}^n \circ (X^n)^{-1} \}$ 
and $\{ \mathbb{P}^n \circ (K^n)^{-1} \}$.  
Using condition (iii) of Definition 6.8 in \cite{ced}, we verify uniform $L^{q'}$-boundedness and tightness of these sequences.

\textbf{Step 2.1:} For $q \geq 1$ with $2q \in [2, q']$, 
we show that 
$\mathbb{E}^{\mathbb{P}^n}(\sup_{0 \le t \le T}|X^n_t|^{2q})$ 
and 
$\mathbb{E}^{\mathbb{P}^n}(\sup_{0 \le t \le T}|K^n_t|^{2q})$ 
are uniformly bounded.  

Applying Itô’s formula and the Skorokhod condition gives
\begin{align*}
|X^n_t|^2 
&= |X^n_0|^2 
+ 2\int_0^t \int_U b(s, X^n_s, \mu^n_s, u) X^n_s Q^n_s(\mathrm{d}u)\, \mathrm{d}s
+ 2\int_0^t \int_U \sigma(s, X^n_s, \mu^n_s, u) X^n_s M^n(\mathrm{d}u, \mathrm{d}s) \\
&\quad + \int_0^t\int_U \sigma^2(s, X^n_s, \mu^n_s, u) Q^n_s(\mathrm{d}u)\, \mathrm{d}s.
\end{align*}
Hence, for some $C>0$,
\begin{align*}
|X^n_t|^{2q} 
\le C \Big(& |X^n_0|^{2q}
+ \Big|\int_0^t\int_U b(s, X^n_s, \mu^n_s, u) X^n_s Q^n_s(\mathrm{d}u)\, \mathrm{d}s\Big|^q
+ \Big|\int_0^t \int_U \sigma(s, X^n_s, \mu^n_s, u) X^n_s M^n(\mathrm{d}u, \mathrm{d}s)\Big|^q \\
& + \Big|\int_0^t \int_U \sigma^2(s, X^n_s, \mu^n_s, u) Q^n_s(\mathrm{d}u)\, \mathrm{d}s\Big|^q \Big).
\end{align*}
Taking the supremum over $[0, T]$ and expectations, applying the BDG inequality, and using 
$2ab \le \frac{a^2}{\varepsilon} + \varepsilon b^2$ together with the growth bounds on $b$ and $\sigma$, we derive
\begin{align*}
\mathbb{E}^{\mathbb{P}^n}\left(\sup_{0 \le t \le T}|X^n_t|^{2q}\right)
&\le C\left(
1 + \mathbb{E}^{\mathbb{P}^n}(|X^n_0|^{2q})
+ \int_0^T \Big(\int_{\mathbb{R}^+} |z|^2 \mu^n_t(\mathrm{d}z)\Big)^{q} \mathrm{d}t
+ \mathbb{E}^{\mathbb{P}^n}\left(\int_0^T |X^n_s|^{2q}\mathrm{d}s\right)
\right)\\
&\quad + \tfrac{1}{2}\mathbb{E}^{\mathbb{P}^n}\left(\sup_{0 \le s \le T}|X^n_s|^{2q}\right).
\end{align*}  
Since $\mu^n \to \mu$ in $\mathcal{P}_{2}(\mathcal{C}^+)$, it follows that $\int_{\mathbb{R}^+} |z|^2 \, \mu_t^n(\mathrm{d}z)$ is uniformly bounded in $n$. Applying Gronwall’s inequality and using that $\lambda \in \mathcal{P}_{q'}(\mathbb{R}^+)$, we obtain the desired estimate,

\begin{equation}\label{eqq22}
\mathbb{E}^{\mathbb{P}^n}\left[\sup_{0 \le t \le T}|X^n_t|^{2q}\right] \le C
\end{equation}
from
\begin{equation}\label{k^n}
K^n_t = X^n_t - X^n_0 - \int_U b(t, X^n_t, \mu^n_t, u) Q^n_t(\mathrm{d}u)\, \mathrm{d}t 
- \int_U \sigma(t, X^n_t, \mu^n_t, u) M^n(\mathrm{d}u, \mathrm{d}t),
\end{equation}
by similar arguments, we obtain uniform boundedness of 
$\mathbb{E}^{\mathbb{P}^n}\left(\sup_{0 \le t \le T}|K^n_t|^{2q}\right)$.\\

\textbf{Step 2.2:} We now establish the tightness of the sequence in $\mathcal{C}$. 
According to Aldous’s criterion (see \cite{belg}, Thm.~16.10), it is sufficient to verify that
\begin{equation*}
\lim_{\delta \downarrow 0} 
\sup_{n} \sup_{\tau \in \mathcal{T}}
\mathbb{E}^{\mathbb{P}^n}\!\left( \left| Z_{(\tau+\delta)\wedge T} - Z_{\tau} \right|^2 \right) = 0 \quad \text{for } Z=\{X^n, K^n\},
\end{equation*}
where $\mathcal{T}$ denotes the set of all $[0,T]$-valued stopping times.

Fix $\tau \in \mathcal{T}$, and let $(X^{n,\tau}, K^{n,\tau})$ denote the solution to the reflected SDE driven by $(Q^n, M^n)$ with initial condition $X^n_\tau$.
By the same arguments as before,
\begin{equation}\label{eqq33}
\mathbb{E}^{\mathbb{P}^n}\left(\sup_{0 \le t \le T}(X^{n,\tau}_t)\right) \le C.
\end{equation}
Applying Itô’s formula to $|X^n_t - X^{n,\tau}_t|^2$ gives
\begin{align*}
\mathrm{d}|X^n_t - X^{n,\tau}_t|^2
&= 2(X^n_t - X^{n,\tau}_t)\Big(\int_{U}\big(b(t,X^n_t, \mu^n_t, u) - b(t,X^{n,\tau}_t, \mu^n_t,u)\big)Q^n_t(\mathrm{d}u)\, \mathrm{d}t\Big) \\
&\quad + 2(X^n_t - X^{n,\tau}_t)\!\int_U \Big(\sigma(t,X^n_t,\mu^n_t,u) - \sigma(t,X^{n,\tau}_t,\mu^n_t,u)\Big) M^n(\mathrm{d}u,\mathrm{d}t) \\
&\quad + \int_U \Big(\sigma(t,X^n_t,\mu^n_t,u) - \sigma(t,X^{n,\tau}_t,\mu^n_t,u)\Big)^2 Q^n_t(\mathrm{d}u)\, \mathrm{d}t \\
&\quad + 2(X^n_t - X^{n,\tau}_t)\,\mathrm{d}K^n_t - 2(X^n_t - X^{n,\tau}_t)\,\mathrm{d}K^{n,\tau}_t.
\end{align*}
By Skorokhod’s condition, the last term is non-positive.  
Using the Lipschitz continuity of $b$ and $\sigma$ assumption \hyperlink{A.2}{(A.2)}, we get
\begin{equation*}
\mathbb{E}^{\mathbb{P}^n}\left(|X^n_{(\tau+\delta)\wedge T} - X^{n,\tau}_{(\tau+\delta)\wedge T}|^2\right)
\le C \mathbb{E}^{\mathbb{P}^n} \Big(\!\int_\tau^{(\tau+\delta)\wedge T}
|X^n_t - X^{n,\tau}_t|^2 \, \mathrm{d}t \Big).
\end{equation*}
Hence,
\begin{equation*}
\mathbb{E}^{\mathbb{P}^n}\left(|X^n_{(\tau+\delta)\wedge T} - X^{n,\tau}_{(\tau+\delta)\wedge T}|^2\right)
\le 2C\delta
\left(\mathbb{E}^{\mathbb{P}^n}\big[\sup_t|X^n_t|^2\big] 
+ \mathbb{E}^{\mathbb{P}^n}\big[\sup_t|X^{n,\tau}_t|^2\big]\right).
\end{equation*}
By uniqueness of the reflected SDE~\eqref{tight1}, $X^{n,\tau} = X^n_{\cdot \wedge \tau}$, then
\begin{equation}\label{eqq444}
\mathbb{E}^{\mathbb{P}^n}\left(|X^n_{(\tau+\delta)\wedge T} - X^n_\tau|^2\right)
\le 4C\delta\, \mathbb{E}^{\mathbb{P}^n}\left(\sup_{0 \le t \le T}|X^n_t|^2\right)
\le l(\delta),
\end{equation}
where $l(\delta) \to 0$ as $\delta \to 0$.  
An analogous argument using~\eqref{k^n} yields the tightness of $(K^n)_n$.\\

\textbf{Step 2.3:} Let $P$ be any weak limit of $(P^n)$ in $\mathcal{P}_{2}(\Omega)$.  
Then $P \circ (X_0)^{-1} = \lambda$ and $\int_0^T X_t\, \mathrm{d}K_t = 0$ $P$-a.s.  

For $0 \le s \le t \le T$, $\psi \in \mathcal{C}_0^{\infty}$, and bounded continuous $\mathcal{F}_s$-measurable $H$, define
\begin{equation*}
\begin{aligned}
&\mathbb{E}^{P^n}\left[\left(\int_0^t \mathcal{L}\psi(s,X_s,\mu^n_s,u)\, Q_s(\mathrm{d}u)\, \mathrm{d}s\right) H\right]
- \mathbb{E}^{P}\left[\left(\int_0^t \mathcal{L}\psi(s,X_s,\mu_s,u)\, Q_s(\mathrm{d}u)\, \mathrm{d}s\right) H\right] \\
&= \mathbb{E}^{P^n}\left[ \left( \int_0^t \mathcal{L}\psi(s,X_s,\mu^n_s,u)\, Q_s(\mathrm{d}u)\, \mathrm{d}s - \int_0^t \mathcal{L}\psi(s,X_s,\mu_s,u)\, Q_s(\mathrm{d}u)\, \mathrm{d}s \right) H \right]\\ &+\mathbb{E}^{P^n}\left[ \left( \int_0^t \mathcal{L}\psi(s,X_s,\mu_s,u)\, Q_s(\mathrm{d}u)\, \mathrm{d}s \right) H \right]-\mathbb{E}^{P}\left[ \left( \int_0^t \mathcal{L}\psi(s,X_s,\mu_s,u)\, Q_s(\mathrm{d}u)\, \mathrm{d}s \right) H \right]\\
&= J^n_1 + J^n_2.
\end{aligned}
\end{equation*}
From \hyperlink{A.2}{(A.2)} and \hyperlink{A.3}{(A.3)}, we have the following estimate for $J^n_1$
\begin{align*}
|J^n_1|
&\le C_{\psi'} \int_0^t W_{\mathbb{R}^+,2}(\mu^n_s,\mu_s)\, \mathrm{d}s
+ C_{\psi''}\int_0^t (1+F(W_{\mathbb{R}^+,2}(\mu^n_s,\delta_0),W_{\mathbb{R}^+,2}(\mu_s,\delta_0)))\, 
W_{\mathbb{R}^+,2}(\mu^n_s,\mu_s)\, \mathrm{d}s \\
&\le C_{\psi'} \int_0^t W_{\mathbb{R}^+,2}(\mu^n_s,\mu_s)\, \mathrm{d}s\\
&+ C_{\psi''}\left(\int_0^t (1+F(W_{\mathbb{R}^+,2}(\mu^n_s,\delta_0),W_{\mathbb{R}^+,2}(\mu_s,\delta_0))^2 \mathrm{d}s\right)^{1/2}
\left(\int_0^t W_{\mathbb{R}^+,2}(\mu^n_s,\mu_s)^2 \mathrm{d}s\right)^{1/2}.
\end{align*}
Since $\sup_n W_{\mathbb{R}^+,2}(\mu^n_t,\delta_0)<\infty$ and $F$ is locally bounded, Lemma~\ref{sumconv} implies $J^n_1 \to 0$. Using the same reasoning as in \textbf{Step 1}, we get $J^n_2 \to 0$.\\
Moreover, the mapping 
$(x,k) \mapsto \psi(x_t) - \int_0^t \psi'(x_s)\,\mathrm{d}k_s$
is continuous and satisfies
$$|\psi(x_t) - \int_0^t \psi'(x_s)\,\mathrm{d}k_s| \le C_{\psi,\psi'}(1+\|k\|_T^2).$$  
By (iv) of Theorem~\ref{convlaw},
\begin{equation*}
\lim_{n \to \infty}
\mathbb{E}^{P^n}\left[\left(\psi(X_t)-\int_0^t \psi'(X_s)\,\mathrm{d}K_s\right)H\right]
= 
\mathbb{E}^{P}\left[\left(\psi(X_t)-\int_0^t \psi'(X_s)\,\mathrm{d}K_s\right)H\right].
\end{equation*}
By combining both limits, we obtain
\begin{equation*}
0 = \lim_{n \to \infty}
\mathbb{E}^{P^n}\big[(\mathcal{M}_t^{\mu^n,\psi}-\mathcal{M}_s^{\mu^n,\psi})H\big]
= 
\mathbb{E}^{P}\big[(\mathcal{M}_t^{\mu,\psi}-\mathcal{M}_s^{\mu,\psi})H\big].
\end{equation*}
Therefore, $P$ satisfies the martingale property defining $\mathcal{R}(\mu)$, and thus 
\begin{equation*}
P \in \mathcal{R}(\mu).
\end{equation*}
We therefore conclude that $\mathcal{R}$ is upper hemicontinuous. 
Moreover, the compactness of $\mathcal{R}(\mu)$ for each $\mu$ follows from the same arguments as in \textbf{Step~2}.
\end{proof}
\begin{lemma} \label{contcost}
Under Assumption \ref{ass.A}, the mapping 
$ J: \mathrm{Gr}(\mathcal{R}) \to \mathbb{R} $
is continuous.
\end{lemma}
\begin{proof}
Let $\{(\mu^n, P^n)\}_{n \ge 1} \subset \mathrm{Gr}(\mathcal{R})$ be a sequence such that 
$P^n \in \mathcal{R}(\mu^n)$ for each $n$, and 
\begin{equation*}
(\mu^n, P^n) \longrightarrow (\mu, P)
\quad \text{in } 
\mathcal{P}_2(\mathcal{C}^+) \times \mathcal{P}_2(\Omega),
\end{equation*}
where $P \in \mathcal{R}(\mu)$.
We show that $J(\mu^n, P^n) \to J(\mu, P)$.\\
We can write
\begin{equation*}
|J(\mu^n, P^n) - J(\mu, P)| 
\le |\mathcal{J}^n_1| + |\mathcal{J}^n_2|,
\end{equation*}
where
\begin{align*}
\mathcal{J}^n_1
&:= \mathbb{E}^{P^n}\!\Bigg[
\int_0^T \int_U 
\!\big(f(t, X_t, \mu^n_t, u) - f(t, X_t, \mu_t, u)\big) Q_t(\mathrm{d}u)\, \mathrm{d}t\Bigg]  \\
&\hspace{1.6cm}
+\, \mathbb{E}^{P^n}\!\Bigg[g(X_T, \mu^n_T) - g(X_T, \mu_T)
+ \int_0^T \big(h(t, X_t, \mu^n_t) - h(t, X_t, \mu_t)\big)\, \mathrm{d}K_t
\Bigg] \\[1ex]
\mathcal{J}^n_2
&:= \mathbb{E}^{P^n}\Bigg[
\int_0^T \int_U f(t, X_t, \mu_t, u) Q_t(\mathrm{d}u)\, \mathrm{d}t
+ g(X_T, \mu_T)
+ \int_0^T h(t, X_t, \mu_t)\, \mathrm{d}K_t
\Bigg] \\
&\quad -
\mathbb{E}^{P}\Bigg[
\int_0^T \int_U f(t, X_t, \mu_t, u) Q_t(\mathrm{d}u)\, \mathrm{d}t
+ g(X_T, \mu_T)
+ \int_0^T h(t, X_t, \mu_t)\, \mathrm{d}K_t
\Bigg].
\end{align*}
According to Lemma \ref{continuity-hkk} and Corollary A.5 in \cite{laker}, the mapping  
\begin{equation*}
(x,k,q) \longmapsto 
\int_0^T \int_U f(t,x_t,\mu_t,u)\, q_t(\mathrm{d}u)\, \mathrm{d}t 
+ g(x_T,\mu_T)
+ \int_0^T h(t,x_t,\mu_t)\, \mathrm{d}k_t
\end{equation*}
is continuous and satisfies a linear growth bound.
\begin{equation*}
\Big|
\int_0^T \int_U f(t,x_t,\mu_t,u)\, q_t(\mathrm{d}u)\, \mathrm{d}t 
+ g(x_T,\mu_T)
+ \int_0^T h(t,x_t,\mu_t)\, \mathrm{d}k_t
\Big|
\le C\big(1 + \|x\|_T^2 + \|k\|_T^2\big).
\end{equation*}
Hence, the convergence $P^n \to P$ in $\mathcal{P}_2(\Omega)$ implies $\mathcal{J}^n_2 \to 0$.

By the Lipschitz-type dependence of $f$ and $g$ on the measure argument, we have
\begin{align*}
&\mathbb{E}^{P^n}\Bigg[
\int_0^T \int_U |f(t, X_t, \mu_t^n, u) - f(t, X_t, \mu_t, u)|\, Q_t(\mathrm{d}u)\, \mathrm{d}t
+ |g(X_T, \mu_T^n) - g(X_T, \mu_T)|
\Bigg]  \\
&\le C \int_0^T 
\big(1 + F(W_{\mathbb{R}^+,2}(\mu_t^n,\delta_0), W_{\mathbb{R}^+,2}(\mu_t,\delta_0))\big)
W_{\mathbb{R}^+,2}(\mu_t^n,\mu_t)\, \mathrm{d}t.\\
&\le C 
\Bigg(
\int_0^T \big(1 + F(W_{\mathbb{R}^+,2}(\mu_t^n,\delta_0), W_{\mathbb{R}^+,2}(\mu_t,\delta_0))\big)^2 \mathrm{d}t
\Bigg)^{1/2}
\Bigg(
\int_0^T W_{\mathbb{R}^+,2}(\mu_t^n,\mu_t)^2 \mathrm{d}t
\Bigg)^{1/2}.
\end{align*}
Since $\sup_n W_{\mathbb{R}^+,2}(\mu_t^n,\delta_0) < \infty$ and $F$ is locally bounded, 
while Lemma~\ref{sumconv} ensures that 
$\int_0^T W_{\mathbb{R}^+,2}(\mu_t^n,\mu_t)^2 \mathrm{d}t \to 0$, 
the entire right-hand side converges to $0$.

Next, by the Lipschitz-type dependence of $h$ on the measure argument, we have
\begin{align*}
&\mathbb{E}^{P^n}\Bigg[\int_0^T |h(t, X_t, \mu_t^n) - h(t, X_t, \mu_t)|\, \mathrm{d}K_t\Bigg] \\
&\le
C\, \mathbb{E}^{P^n}\Bigg[
\int_0^T 
\big(1 + F(W_{\mathbb{R}^+,2}(\mu_t^n,\delta_0), W_{\mathbb{R}^+,2}(\mu_t,\delta_0))\big)
W_{\mathbb{R}^+,2}(\mu_t^n,\mu_t)\, \mathrm{d}K_t
\Bigg].\\
&\leq C \big(1 + F(W_{\mathcal{C}^+,2}(\mu^n,\delta_0), W_{\mathcal{C}^+,2}(\mu,\delta_0))\big)
W_{\mathcal{C}^+,2}(\mu^n,\mu)\, \mathbb{E}^{P^n}(K_T).
\end{align*}
Since $\sup_n W_{\mathcal{C}^+,2}(\mu^n,\delta_0) < \infty$ 
and, by Step~2.1 of Lemma~\ref{lemma-upper}, 
$\sup_n \mathbb{E}^{P^n}(K_T^{q'})^{1/q'} < \infty$, 
the right-hand side converges to $0$ as $n \to \infty$. Then $\mathcal{J}^n_1 \to 0.$

Finally, combining both limits yields
\begin{equation*}
|J(\mu^n, P^n) - J(\mu, P)| \longrightarrow 0,
\end{equation*}
which establishes the continuity of $J$.
\end{proof}
To apply Kakutani’s fixed-point theorem, it is necessary to construct a suitable invariant domain for the correspondence $\Phi$. 
For this purpose, we establish the following standard estimate.

\begin{lemma}\label{lemma-invary}
Under Assumption \ref{ass.A}, let $ q \geq 1 $ be such that $2q \in [2, q'] $. 
Then there exists a constant $ C' = C(T, C_1, q, |\lambda|^{q'}) $ such that, for any 
$ \mu \in \mathcal{P}_2(\mathcal{C}^+) $ and any $ P \in \mathcal{R}(\mu) $,
\begin{equation}\label{est55}
 \mathbb{E}^P \left[ \|X\|^{2q}_{T} \right] 
 \leq C' \left( 1 +  \|\mu \|^{2q}_T \right).
\end{equation}
As a consequence, $ P \in \mathcal{P}_2(\Omega) $. Moreover, if $ \mu = P \circ X^{-1} $, then
\begin{equation} \label{constinvar}
\|\mu\|^{2q}_T = \mathbb{E}^P \left[ \|X\|^{2q}_{T} \right] 
\leq C'' ,
\end{equation}
where $ C'' := C''(C', T) $.
\end{lemma}
\begin{proof}
By Proposition~\ref{rep mart}, there exists a filtered probability space 
$ (\tilde{\Omega}, \tilde{\mathcal{F}}, (\tilde{\mathcal{F}}_t)_{t \in [0,T]}, \mathbb{Q}) $
supporting an $ \tilde{\mathcal{F}}_t $-adapted pair of processes 
$ (\tilde{X}, \tilde{K}) $, together with $ m $ orthogonal 
$ \tilde{\mathcal{F}}_t $-martingale measures 
$ \tilde{M} = (\tilde{M}_1, \dots, \tilde{M}_m) $ on $ U \times [0,T] $
with intensity $ \tilde{Q} $, such that
\begin{equation}\label{foroptimal}
\begin{cases} 
\mathrm{d}\tilde{X}_{t} 
= \displaystyle\int_{U} b(t,\tilde{X}_{t}, \mu_t, u)\,\tilde{Q}_{t}(\mathrm{d}u)\,\mathrm{d}t 
+ \int_{U} \sigma(t,\tilde{X}_{t},\mu_t, u)\,\tilde{M}(\mathrm{d}u, \mathrm{d}t)
+ \mathrm{d}\tilde{K}_{t}, \\[0.6em]
\tilde{X}_{t} \geq 0 \quad \text{a.s.},\\[0.3em]
\displaystyle\int_{0}^{T}  \tilde{X}_{t}\,\mathrm{d}\tilde{K}_{t} = 0 \quad \text{a.s.}
\end{cases}
\end{equation}
and 
$ P \circ (X, Q, K)^{-1} = \mathbb{Q} \circ (\tilde{X}, \tilde{Q}, \tilde{K})^{-1} $.

Following the same reasoning as in Step~2.1 of the proof of Lemma~\ref{lemma-upper}, 
we obtain the estimate
\begin{equation*}
\mathbb{E}^{\mathbb{Q}}\left( \sup_{0 \leq t \leq T} |\tilde{X}_t|^{2q} \right) 
\leq C(q, C_1) \left(
1 + \mathbb{E}^{\mathbb{Q}}(|\tilde{X}_0|^{2q})
+ \int_0^T \|\mu\|^{2q}_t \,\mathrm{d}t
+ \mathbb{E}^{\mathbb{Q}}\left( \int_0^T |\tilde{X}_t|^{2q}\,\mathrm{d}t \right)
\right).
\end{equation*}
Applying Gronwall’s inequality yields the desired estimate~\eqref{est55}.

If, in addition, $ \mu = P \circ X^{-1} = \mathbb{Q} \circ \tilde{X}^{-1} $, then
\begin{equation}
\|\mu \|^{2q}_T 
= \mathbb{E}^{\mathbb{Q}}\left( \sup_{0 \leq t \leq T} |\tilde{X}_t|^{2q} \right)
\leq C' \,\mathbb{E}^{\mathbb{Q}}\left(\int_0^T \big( 1 + \|\mu \|^{2q}_t \big)\,\mathrm{d}t \right),
\end{equation}
and the second assertion follows by another application of Gronwall’s lemma.
\end{proof}

\begin{proof}[Proof of Theorem~\ref{theomfg}]
By Lemmas~\ref{lemma-lower}, \ref{lemma-upper}, and~\ref{contcost}, 
the correspondence $ \mathcal{R} $ is continuous, has nonempty and compact values, 
and the mapping $J : \mathrm{Gr}(\mathcal{R}) \to \mathbb{R} $ is jointly continuous. 
Hence, by the Berge Maximum Theorem~\cite[Theorem~17.31]{border}, 
the optimal correspondence $ \mathcal{R}^* $ is nonempty and upper hemicontinuous. 

Since $ \mathcal{R}(\mu) $ is convex for each $ \mu $, and the mapping 
$ P \mapsto J(\mu, P) $ is linear, it follows that 
$ \mathcal{R}^*(\mu) $ is convex for all 
$ \mu \in \mathcal{P}_2(\mathcal{C}^+) $. 
Moreover, the mapping 
$ \mathcal{P}_2(\Omega) \ni P \mapsto P \circ X^{-1} \in \mathcal{P}_2(\mathcal{C}^+) $ 
is linear and continuous. Consequently, the induced set-valued map 
$ \Phi $ is upper hemicontinuous with nonempty, compact, and convex values.  

To apply a fixed-point theorem, we construct a convex and compact subset of 
$ \mathcal{P}_2(\mathcal{C}^+) $ that is invariant under $ \Phi $. 
Define
\begin{equation*}
S = \left\{ 
P \in \mathcal{P}_2(\mathcal{C}^+) : 
\begin{array}{l}
\text{for any stopping time } \tau, \,
\mathbb{E}^{P}\big(|X_{(\tau+\delta)\wedge T} - X_{\tau}|\big) \leq l(\delta), \\[0.4em]
\mathbb{E}^P\left( \sup_{0 \leq t \leq T} |X_t|^{2q} \right) \leq C''
\end{array}
\right\},
\end{equation*}
where $ l(\delta) $ and $ C'' $ are defined in~\eqref{eqq444} and~\eqref{constinvar}, respectively. \\
By~\cite[Theorem~16.10]{belg}, the set $ S $ is relatively compact in 
$ \mathcal{P}_2(\mathcal{C}^+) $. It is nonempty and convex, 
and by Lemma~\ref{lemma-invary}, we have 
$ \Phi(\mu) \subset S \subset \overline{S} $ for all 
$ \mu \in \mathcal{P}_2(\mathcal{C}^+) $. 
Hence, $ \Phi : \overline{S} \to 2^{\overline{S}} $ 
is nonempty-valued and upper hemicontinuous.  

As in the proof of Theorem~4.1 in~\cite{laker}, 
we embed $ \Phi $ from $ \mathcal{P}_2(\mathcal{C}^+) $ into  
$ \mathcal{M}(\mathcal{C}([0,T],\mathbb{R})) $, 
the space of bounded signed measures on $ \mathcal{C}([0,T],\mathbb{R}) $ 
equipped with the topology of weak convergence. 
This space is locally convex and Hausdorff.  

Applying the Kakutani–Fan–Glicksberg fixed-point theorem 
(\cite[Corollary~17.55]{border}) to the correspondence $ \Phi$
yields the existence of a fixed point, which, by Definition~\ref{definition2}, 
corresponds to a relaxed MFG solution.
\end{proof}

\section*{Appendix}
\begin{theorem}(Theorem 7.12 of \cite{villani2}).\label{convlaw}
Let $(\mathbb{D}, d)$ be a Polish space, and $p \in [1, \infty)$. Consider a sequence of probability measures $(\mu_n)_{n \in \mathbb{N}}$ in $\mathcal{P}_p(\mathbb{D})$, and let $\mu$ be another element of $\mathcal{P}_p(\mathbb{D})$. The sequence $(\mu_n)$ is said to converge weakly in $\mathcal{P}_p(\mathbb{D})$ if any of the following equivalent properties is satisfied for some (and thus for any) $x_0 \in \mathbb{D}$:
\begin{itemize}
\item[(i)] $\mu_n \rightharpoonup \mu$ (i.e., $\mu_n$ converges weakly to $\mu$) and 
\begin{equation*}
\int d(x_0, x)^p \, \mathrm{d}\mu_n(x) \to \int d(x_0, x)^p \, \mathrm{d}\mu(x).
\end{equation*} 
\item[(ii)] $W_{\mathbb{D}, p}(\mu^n, \mu) \to 0.$ 
\item[(iii)] $\mu_n \rightharpoonup \mu$ and 
\begin{equation*}
\lim_{R \to \infty} \limsup_{n \to \infty} \int_{d(x_0, x) \geq R} d(x_0, x)^p \, \mathrm{d}\mu_n(x) = 0.
\end{equation*}
\item[(iv)] For all continuous functions $ \varphi : \mathbb{D} \to \mathbb{R}$ satisfying  
\begin{equation*}
|\varphi(x)| \leq C(1 + d(x_0, x)^p), \quad C \in \mathbb{R}^+,
\end{equation*}  
it holds that  
\begin{equation*}
\int_{\mathbb{D}} \varphi(x) \, \mathrm{d}\mu_n(x) \to \int_{\mathbb{D}} \varphi(x) \, \mathrm{d}\mu(x) \quad \text{as } n \to \infty.
\end{equation*}
\end{itemize}
\end{theorem}
\begin{lemma} \label{sumconv}
Let $(\mu^n)_{n \in \mathbb{N}} \subset \mathcal{P}_p(\mathcal{C})$ be a sequence of probability measures converging to $\mu \in \mathcal{P}_p(\mathcal{C})$, where $\mathcal{C}$ is the space of all continuous functions from $ [0, T]$ to $\mathbb{R}$, endowed with the topology of uniform convergence.
. Then, for any $q \geq 1$,
\begin{equation*}
\int_0^T W_{\mathcal{C}, p}(\mu^n_{t}, \mu_t)^q \, \mathrm{d}t \to 0 \quad \text{as} \quad n \to \infty.
\end{equation*}
\end{lemma}
\begin{proof}
This lemma can be proven in the same manner as \cite[Lemma A.3]{luciano}.
\end{proof}
\begin{remark} \label{convterminal}
If $W_{\mathcal{C},p}(\mu^n, \mu) \to 0$, then $\mu^n_{t} \to \mu_t$ for every $t \in [0, T]$. \\
Indeed, since any continuous function $ \varphi$ on $\mathbb{R}$ that satisfies $|\varphi(x)| \leq C(1 + |x|^p)$ also satisfies
\begin{equation*}
|\varphi(x_t)| \leq C(1 + |x_t|^p) \leq C(1 + \|x\|_T^p), \quad x \in \mathcal{C},
\end{equation*}
and the mapping $\mathcal{C} \ni x \mapsto \varphi(x_t) \in \mathbb{R}$ is continuous for every $t$. \\
Thus, by point (iv) of Theorem \ref{convlaw}, we have
\begin{equation*}
\left| \int_{\mathbb{R}} \varphi(x) \, \mu_t^n(\mathrm{d}x) - \int_{\mathbb{R}} \varphi(x) \, \mu_t(\mathrm{d}x) \right|
=
\left| \int_{\mathcal{C}} \varphi(x_t) \, \mu^n(\mathrm{d}x) - \int_{\mathcal{C}} \varphi(x_t) \, \mu(\mathrm{d}x) \right| \to 0.
\end{equation*}
\end{remark}

\begin{lemma}\label{continuity-hkk}
Let $h : [0,T]\times \mathbb{R}\times \mathcal{P}_2(\mathbb{R})\to\mathbb{R}$ 
be continuous and satisfy the linear growth condition
\begin{equation*}
|h(t,x,\mu)| \le C\big(1 + |x| + \big(\int_{\mathbb{R}} |y|^2 \mu(\mathrm{d}y)\big)^{1/2}\big),
\qquad (t,x,\mu)\in [0,T]\times\mathbb{R}\times\mathcal{P}_2(\mathbb{R}).
\end{equation*}
Then the functional
\begin{equation*}
\Phi(x,k,\mu)
:=
\int_0^T h(t,x_t,\mu_t)\,\mathrm{d}k_t
\end{equation*}
is continuous on 
$\mathcal{C}([0,T];\mathbb{R})
\times \mathcal{A}([0,T];\mathbb{R})
\times \mathcal{P}_2(\mathcal{C}([0,T];\mathbb{R}))$,
where the first two spaces are equipped with the uniform topology 
and the last with the Wasserstein topology.
\end{lemma}

\begin{proof}
Let $(x^n,k^n,\mu^n)\to(x,k,\mu)$ uniformly on $[0,T]$.
Since $h$ is continuous, we have for each $t\in[0,T]$,
\begin{equation*}
h(t,x^n_t,\mu^n_t)\to h(t,x_t,\mu_t).
\end{equation*}
By the Heine–Cantor theorem, $h$ is uniformly continuous on compact subsets, hence
\begin{equation*}
\sup_{0 \leq t \leq T} |h(t,x^n_t,\mu^n_t)-h(t,x_t,\mu_t)| \to 0.
\end{equation*}
We can therefore write
\begin{equation*}
\Phi(x^n,k^n,\mu^n)
- \Phi(x,k,\mu)
=
\int_0^T\big(h(t,x^n_t,\mu^n_t)-h(t,x_t,\mu_t)\big)\,\mathrm{d}k^n_t
+
\int_0^T h(t,x_t,\mu_t)\,\mathrm{d}(k^n_t-k_t)
=: I_1^n + I_2^n.
\end{equation*}

For the first term, we have
\begin{equation*}
|I^n_1| \leq  \sup_{t\in[0,T]}|h(t,x^n_t,\mu^n_t)-h(t,x_t,\mu_t)| \, |k^n_T|.
\end{equation*}
Since $(k^n)$ converges uniformly to $k$, the sequence $(|k^n_T|)$ is bounded, and hence $I^n_1 \to 0$.

For the second term, observe that $h(\cdot,x_\cdot,\mu_\cdot)$ is continuous on $[0,T]$, and therefore bounded.  
Since $k^n \to k$ uniformly, it follows from the convergence of Riemann–Stieltjes integrals that
\begin{equation*}
\int_0^T h(t,x_t,\mu_t)\,\mathrm{d}k^n_t
\longrightarrow
\int_0^T h(t,x_t,\mu_t)\,\mathrm{d}k_t,
\end{equation*}
that is, $I_2^n \to 0$.

Combining both terms, we obtain
\begin{equation*}
\Phi(x^n,k^n,\mu^n)
\longrightarrow
\Phi(x,k,\mu),
\end{equation*}
which establishes the desired continuity.
\end{proof}

\begin{definition}[\cite{border}, Chapter~17]\label{def of set-valued conti}
Let $\mathcal{E}$ and $\mathcal{G}$ be two metric spaces.  
A set-valued mapping $\mathcal{H} : \mathcal{E} \rightrightarrows \mathcal{G}$ is defined as follows:

\begin{description}
\item[i)] $\mathcal{H}$ is said to be \emph{lower hemicontinuous} if,  
for every sequence $(x_n) \subset \mathcal{E}$ with $x_n \to x$ and every $y \in \mathcal{H}(x)$, there exists a subsequence $(x_{n_k})$ and elements $y_{n_k} \in \mathcal{H}(x_{n_k})$ such that $y_{n_k} \to y$.

\item[ii)] If $\mathcal{H}(x)$ is closed for each $x \in \mathcal{E}$, then $\mathcal{H}$ is said to be \emph{upper hemicontinuous} if,  
whenever $x_n \to x$ in $\mathcal{E}$ and $y_n \in \mathcal{H}(x_n)$ for all $n$, the sequence $(y_n)$ admits a limit point in $\mathcal{H}(x)$.
\end{description}

The correspondence $\mathcal{H}$ is said to be \emph{continuous} if it is both upper and lower hemicontinuous.
\end{definition}

\bibliographystyle{abbrv}
\bibliography{games}
\end{document}